\documentclass[final]{siamltex}
\usepackage{}
\usepackage{graphicx}
\usepackage{multirow,bigstrut}
\usepackage{amsmath,amsfonts,amssymb}
\usepackage{mathrsfs}
\usepackage{color}
\usepackage{upgreek}
\usepackage{bm}
\usepackage{verbatim}
\usepackage{mathtools}
\usepackage{calligra}
\usepackage[T1]{fontenc}
\usepackage{subfigure}

\usepackage{verbatim}
\usepackage{exscale}
\usepackage{relsize}
\newtheorem{rem}{\hspace{1mm}Remark}[section]

\usepackage{tikz}
\usetikzlibrary{shapes,arrows}
 
\begin{document}
 
% --------------------------------------------------------------
%                         Start here
% --------------------------------------------------------------
 
\title{Mathematical and numerical analysis of a nonlocal Drude model in nanoplasmonics}%replace with the appropriate homework number
\author{ Chupeng Ma\thanks{{Department of Applied Mathematics,
The Hong Kong Polytechnic University, Kowloon, Hong Kong, China; ({\tt machupeng@lsec.cc.ac.cn}).}}
\and Yongwei Zhang\thanks{{Faculty of Mathematics and Statistics, Zhengzhou University, Zhengzhou, China; ({\tt zhangyongwei@lsec.cc.ac.cn}).}}
\and Jun zou\thanks{{Department of Mathematics, The Chinese University of Hong Kong, Shatin, Hong Kong, China; ({\tt zou@math.cuhk.edu.hk}).}}}
 
\maketitle
\begin{abstract}
In this paper, we consider the frequency-domain Maxwell's equations coupled to a nonlocal Drude model which describes the nonlocal optical response in metallic nanostructures. We prove the existence and uniqueness of weak solutions to the coupled equations. A Galerkin finite element method based on the Raviart--Thomas and N\'{e}d\'{e}lec elements is proposed to solve the equations and the associated error estimates are given. This is the first work on the mathematical and numerical analysis of this model. Numerical examples are presented to verify our theoretical analysis.
\end{abstract}



\section{Introduction}\label{sec-1}
Nanoplasmonics is an active research field concerned with the study of optical properties of metallic nanostructures \cite{Duan,Stock}. The interaction of metallic nanostrucutes with light at optical frequencies produces the excitation of (localized) surface plasmons, i.e., the collective oscillation of conduction band electrons at the metal surface, leading to many unusual and fascinating properties. It enables the confinement of light at the nanoscale vicinity of metal surfaces, the enormous local fields enhancement of the incident wave, and the squeezing of light beyond the diffraction limit, providing unparalleled means for manipulation of light at the nanoscale. As a result, nanoplasmonics has found applications in different fields, such as near-field scanning microscopy \cite{Lucas}, ultrasensitive sensing and detection \cite{Stewart}, and plasmonic waveguilding \cite{Oulton}. To be able to understand and make use of nanoplasmonic phenomena, an appropriate modeling for describing the optical response of metallic nanostructures is required. 

Historically, the interaction of light with metals has been described by the classical theory of light-matter interaction in which light and matter (the free electrons) are described by Maxwell's equations and Newtonian mechanics, respectively. The most widely used model to describe the optical response of metals is Drude's model \cite{Drude}. In this model, the collective oscillation of conduction electrons in metals subject to driving optical fields is analysed within the framework of local-response approximation (LRA), where the material response occurs only in the point of space of the perturbation and there is no response at even short distances. Drude's model has achieved a great success in the modeling of bulk metals. However, as the size of metallic structures shrinks down to nanometer scales, nonlocal interaction effects between electrons become predominant and Drude's model is inadequate to explain experimentally observable phenomena.

In view of the limitation of Drude's model, the nonlocal response theories for metallic nanostructures have gained considerable interest and attention in the past decade and some improved models have been proposed, including the hydrodynamic model \cite{Ginzburg}, the nonlocal hydrodynamic Drude (NHD) model \cite{Raza}, and the generalized nonlocal optical response (GNOR) model \cite{Mort}. In the hydrodynamic model, the free electrons in metals are modeled as a charged fluid and described by hydrodynamic equations of Euler-type. The NHD and GNOR models are derived from the hydrodynamic model by the linear-response approximation. All these models are coupled to Maxwell's equations (in both frequency domain and time domain) and form coupled systems of PDEs. 

Due to a relatively simple form and the successful interpretation of observable nonlocal effects \cite{Cira}, the NHD model has drew much attention in recent years and become a popular model in the study of optical properties of metallic nanostructures. Meanwhile, numerical methods for solving Maxwell's equations coupled to the NHD model have been extensively studied. In \cite{Hire}, the frequency-domain NHD model (frequency-domain Maxwell's equations coupled to the NHD model) is solved for modeling nano-plasmonic structures with complex geometries by using the N\'{e}d\'{e}lec elements based finite element method. In \cite{zheng}, a computational scheme based on the boundary integral equation and method of moments is developed for the frequency-domain NHD model to predict the interaction of light with metallic nanoparticles. The discontinuous Galerkin methods for the time-domain and frequency-domain NHD models have been considered in \cite{Li,Schmitt,Vidal}. For other more numerical methods for this system of PDEs, we refer the reader to \cite{Eremin,Huang,Schmitt-2,Tru} and references therein. 

However, up to now, most existing studies on the NHD model focus on the development of numerical methods or the analysis of physical effects. The theoretical and numerical analysis of this model available in the literature is very limited. In \cite{Huang}, the well-posedness and the stability and convergence of numerical methods are proved for the modified time-domain NHD model. To the best of our knowledge, there seems to be no results on the mathematical and numerical analysis of the frequency-domain NHD model. In fact, just as it is for the frequency-domain and time-domain Maxwell's equations, the proof of the well-posedness and numerical convergence for the frequency-domain NHD model is much difficult than its time-domain counterpart.

In this paper, we present a rigorous mathematical and numerical analysis of the frequency-domain NHD model for the first time. The existence and uniqueness of weak solutions to the equations are proved. A finite element method based on the Raviart--Thomas and N\'{e}d\'{e}lec elements is developed for the equations and the convergence is proved. There are two aspects of this model that make the analysis challenging. First, the curl and div operators both have a large null space in the continuous and discrete levels, which must be removed from the function spaces by using the (discrete) Hemtholtz decompositions. To this end, an understanding of some properties of the continuous function spaces and the finite element spaces is required. Second, the bilinear forms in the weak formulation of the equations are not coercive, which brings difficulty to the analysis in both continuous and discrete levels. To overcome this problem, in the proof of the well-posedness, we first show the uniqueness of weak solutions and then apply the Fredholm alternative to prove the existence of weak solutions, while in the proof of the convergence of the finite element discretization, the theory of convergence of collectively compact operators developed in \cite{Kress,Monk} is used. Although our methods are somewhat similar to those used in \cite{Monk} for the analysis of frequency-domain Maxwell's equations, the proof presented in this paper is much more delicate due to the coupled system nature.

The rest of this paper is organized as follows. In section 2, we give a brief derivation of the NHD model and describe the problem considered in this paper. In section 3, we prove the existence and uniqueness of weak solutions to the equations. In section 4, we propose a finite element discretization for the equations and prove the convergence of the scheme. In section 5, we give some numerical examples to confirm our theoretical analysis.

\section{Nonlocal hydrodynamic Drude model}\label{sec-2}
In this section, we briefly introduce the NHD model and give the problem considered in this paper.

In the absence of external charge and current, macroscopic Maxwell's equations for metals can be written as 
 \begin{equation}\label{eq:1-1}
\begin{array}{@{}l@{}}
{\displaystyle  \nabla\times{\bf E} = -\partial_{t}{\bf B}, \qquad \nabla \cdot {\bf B} = 0,}\\[2mm]
 {\displaystyle \nabla\times{\bf H} = \partial_{t}{\bf D} +{\bf J}, \quad \nabla \cdot {\bf D} = \rho.} 
\end{array}
\end{equation}
The equations link four macroscopic fields ${\bf E}$ (the electric field), ${\bf H}$ (the magnetic field), ${\bf D}$ (the dielectric displacement), and ${\bf B}$ (the magnetic flux density) with the free charge and current densities $\rho$ and ${\bf J}$. Maxwell's equations (\ref{eq:1-1}) are supplemented by the constitutive laws which link ${\bf B}$ to ${\bf H}$ and ${\bf D}$ to ${\bf E}$ via
\begin{equation}\label{eq:1-2}
{\bf B} = \mu \mathbf{H},\quad {\bf D} = \epsilon_{0}\epsilon_{\infty} {\bf E}.
\end{equation}
Here $\mu$ is the magnetic permeability of metals and $\epsilon_{0}\epsilon_{\infty}$ is the electric permittivity of metals that takes into account the polarization of bound electrons (${\epsilon_{0}}$ is the electric permittivity of vacuum). 

We derive the NHD model starting from the hydrodynamic model within which the free electrons are modeled as a charged fluid with the Euler equations:
 \begin{equation}\label{eq:1-3}
\left\{
\begin{array}{@{}l@{}}
{\displaystyle  \partial_{t}n+\nabla\cdot(n{\bf v})= 0 ,}\\[2mm]
 {\displaystyle m_{e}(\partial_{t} + {\bf v}\cdot \nabla +\gamma){\bf v} = -e({\bf E} + {\bf v}\times {\bf B}) - \frac{\nabla p}{n},}
\end{array}
\right.
\end{equation}
where $e$ is the electron charge, $m_{e}$ is the effective electron mass, $n$ is the electron density, ${\bf v}$ is the hydrodynamic velocity, $p$ is the electron pressure, and $\gamma>0$ is the damping constant. The term $-e({\bf E} + {\bf v}\times {\bf B})$ represents the Lorentz force. The polarization charge and current densities $\rho$ and ${\bf J}$ of the free electrons are given as
\begin{equation}\label{eq:1-4}
\rho = -en, \quad {\bf J} = -en{\bf v}.
\end{equation}

The equations (\ref{eq:1-1})-(\ref{eq:1-4}) form the self-consistent Euler--Maxwell coupled equations. In order to simplify the above equations, we linearize the equations (\ref{eq:1-3}) as in perturbation theory by expanding the physical fields in a non-oscillating term (e.g. the constant equilibrium electron density $n_{0}$) and a small first-order dynamic term. In this spirit, we can write the perturbation expansions for $n({\bf x},t)$ and ${\bf v}({\bf x},t)$
\begin{equation}\label{eq:1-4-0}
n({\bf x},t)\approx n_{0} + n_{1}({\bf x},t), \quad {\bf v}({\bf x},t) \approx {\bf v}_{0} + {\bf v}_{1}({\bf x},t).
\end{equation}
Similar expansions can be written for the electric and magnetic fields. Since in the absence of an external field ${\bf v} = {\bf v}_{0} = {\bf 0}$, the nonlinear terms ${\bf v}\cdot\nabla {\bf v}$ and ${\bf v}\times {\bf B}$ vanish due to the linearization. By using the Thomas-Fermi model for the pressure term in (\ref{eq:1-3}), we can linearize it as 
\begin{equation}\label{eq:1-4-1}
\frac{\nabla p}{n} \approx m_{e}\beta^{2}\frac{\nabla n}{n_{0}},
\end{equation}
where $\beta$ is an important parameter representing the nonlocality related to the Fermi velocity \cite{Boardman}. Using the assumptions above, we get the linearized hydrodynamic equation
 \begin{equation}\label{eq:1-4-2}
\partial_{t}{\bf v}  = \frac{-e}{m_{e}}{\bf E}  - \gamma {\bf v} - \beta^{2}\frac{\nabla n}{n_{0}}
\end{equation}
and the linearized continuity equation
\begin{equation}\label{eq:1-5}
\partial_{t}n+n_{0}\nabla\cdot {\bf v}= 0.
\end{equation}
Differentiating (\ref{eq:1-4-2}) with respect to time $t$, inserting the linearized current density ${\bf J} \approx -en_{0}{\bf v}$ and using (\ref{eq:1-5}), we obtain
\begin{equation}\label{eq:1-6}
\partial_{tt}{\bf J} + \gamma\partial_{t}{\bf J} - \beta^{2}\nabla(\nabla \cdot {\bf J}) - \omega_{p}^{2} \varepsilon_{0} \partial_{t} {\bf E} = 0,
\end{equation}
where $\omega_{p} = \sqrt{n_{0}e^{2}/(m_{e}\varepsilon_{0})}$ is the plasma frequency. Combining (\ref{eq:1-1}), (\ref{eq:1-2}), and (\ref{eq:1-6}), we have Maxwell's equations with the NHD model for metals
 \begin{equation}\label{eq:1-7}
\left\{
\begin{array}{@{}l@{}}
{\displaystyle  \nabla\times{\bf E} = -\mu\partial_{t}{\bf H} ,}\\[2mm]
 {\displaystyle \nabla\times{\bf H} = \varepsilon_{0}\varepsilon_{\infty}\partial_{t}{\bf E} +{\bf J} ,} \\[2mm]
 {\displaystyle \partial_{tt}{\bf J} + \gamma\partial_{t}{\bf J} - \beta^{2}\nabla(\nabla \cdot {\bf J}) - \omega_{p}^{2} \varepsilon_{0} \partial_{t} {\bf E} = 0.}
\end{array}
\right.
\end{equation}
Replacing $\partial_{t}$ with $-{\rm i}\omega$ in (\ref{eq:1-7}) by Fourier transformation in the time domain, where ${\rm i}$ is the imaginary unit and $\omega$ is the angular frequency, and eliminating the magnetic field ${\bf H}$, we get Maxwell's equations with the NHD model in frequency domain
 \begin{equation}\label{eq:1-8}
\left\{
\begin{array}{@{}l@{}}
{\displaystyle  \nabla \times (\mu^{-1}\nabla\times {\bf E}) - \varepsilon_{0}\varepsilon_{\infty}\omega^{2}{\bf E} = {\rm i}\omega {\bf J} ,}\\[2mm]
 {\displaystyle \omega(\omega+{\rm i}\gamma){\bf J} + \beta^{2} \nabla(\nabla\cdot {\bf J}) = {\rm i}\omega \omega^{2}_{p}\varepsilon_{0}{\bf E}.}
\end{array}
\right.
\end{equation}
\begin{rem}
By Fourier transformation in the space domain, we replace $\nabla$ with $-{\rm i}{\bf k}$ in the second equation of (\ref{eq:1-8}), which gives
\begin{equation}
\omega(\omega+{\rm i}\gamma){\bf J} - \beta^{2} {\bf k}^{2} {\bf J}= {\rm i}\omega \omega^{2}_{p}\varepsilon_{0}{\bf E},
\end{equation}
and then we obtain the spatially-dispersive (relative) permittivity for the metal
\begin{equation}
\epsilon(\omega,{\bf k}) = \epsilon_{\infty}- \frac{\omega_{p}^{2}}{\omega(\omega+{\rm i}\gamma)-\beta^{2}{\bf k}^{2}}.
\end{equation}
The parameter $\beta$ represents the level of nonlocality. As $\beta\rightarrow 0$, we recover the classical Drude permittivity $\epsilon(\omega) = \epsilon_{\infty}- {\omega_{p}^{2}}/{(\omega(\omega+{\rm i}\gamma))}$.
\end{rem}

In this paper we consider the following equations
 \begin{equation}\label{eq:1-9}
\left\{
\begin{array}{@{}l@{}}
{\displaystyle  \nabla \times (\mu^{-1}\nabla\times {\bf E}) - \varepsilon \omega^{2} {\bf E} = {\rm i}\omega {\bf J}, \quad {\rm in} \;\; \Omega,}\\[2mm]
 {\displaystyle \omega(\omega+{\rm i}\gamma){\bf J} + \beta^{2} \nabla(\nabla\cdot {\bf J}) = {\rm i}\omega \omega^{2}_{p}\varepsilon_{0}{\bf E}, \quad {\rm in}\;\; \Omega_{s}}
\end{array}
\right.
\end{equation}
with the boundary conditions
 \begin{equation}\label{eq:1-10}
\left\{
\begin{array}{@{}l@{}}
{\displaystyle (\mu^{-1}\nabla\times {\bf E})\times {\bf n} - {\rm i}\omega({\bf n}\times{\bf E})\times{\bf n} = {\bf g},\quad {\rm on}\;\; \partial \Omega,}\\[2mm]
 {\displaystyle {\bf n}\cdot {\bf J} = 0,\quad  {\rm on}\;\; \partial \Omega_{s}.}
\end{array}
\right.
\end{equation}
Here $\Omega$ and $\Omega_{s}$ are the bounded, simply-connected, Lipschitz polyhedron domains in $\mathbb{R}^{3}$ with $\bar{\Omega}_{s}\subset \Omega$. $\Omega$ and $\Omega_{s}$ are shown in Fig~2.1. The magnetic permeability $\mu$ and electric permittivity $\varepsilon$ are two piecewise constant functions in the domain $\Omega$, namely
\begin{equation}\label{eq:1-11}
\mu=\left\{
\begin{array}{l}
 \mu_{1}, \,  \quad \hbox{in $\Omega_{s}$,}\\
 \mu_{2}, \,  \quad \hbox{in $\Omega/\Omega_{s}$,}
\end{array}
\right.
\quad\quad  \varepsilon=\left\{
\begin{array}{l}
 \varepsilon_{1}, \,  \quad \hbox{in $\Omega_{s}$,}\\
 \varepsilon_{2}, \,  \quad \hbox{in $\Omega/\Omega_{s}$,}
\end{array}
\right.
\end{equation}
and $\mu_{i}, \varepsilon_{i}\; (i=1,2)$ are positive constants. 
\begin{rem}
The hard-wall boundary conditions for the current density ${\bf J}$ means that the electrons are confined within the metal and spill-out of electrons in free space is neglected. For Maxwell's equations, we apply the first-order Silver--M\"{u}ller boundary conditions \cite{Stupfel}
\begin{equation}\label{eq:1-12}
{\bf n}\times {\bf E} - {\bf H} = {\bf n}\times {\bf E}^{inc} - {\bf H}^{inc}, \quad {\rm on}\,\, \partial \Omega,
\end{equation}
where ${\bf E}^{inc}$ and ${\bf H}^{inc}$ represent the electromagnetic fields of the incoming light. By substituting ${\bf H} = \frac{1}{{\rm i}\omega\mu} \nabla \times {\bf E}$ into (\ref{eq:1-12}) and denoting ${\rm i}{\omega}({\bf H}^{inc} -{\bf n}\times {\bf E}^{inc})\times {\bf n}$ by ${\bf g}$, we get the boundary conditions (\ref{eq:1-10}) for the electric field ${\bf E}$.
\end{rem}

\begin{figure}
\centering
\begin{tikzpicture}
\draw (0,0) ellipse (100pt and 70pt);
\filldraw[fill=black!15!white,draw=black] (0,0) circle (1.1cm);
\node at  (0, 0)  {${\Omega_{s}}$};
\node at (2.0, 0.1) {$\Omega$ };
\draw[thick][->](-1.8,-0.7)--(-1.1, 0);
\node[above = 1pt,left] at (-1.7,-1.0) {$\partial \Omega_{s}$};
\draw[thick][->](4.5,1.5)--(3.5, 0);
\node[below = 1pt,left] at (5.4,1.6) {$\partial \Omega$};
\draw[ultra thick][->](-1.1,2.9)--(0.2, 1.4);
\node[above= 1pt,left] at (-1.1,2.9) {${\rm Incident \,wave}$};
\end{tikzpicture}
\caption{Sketch of the domain.}
\end{figure}
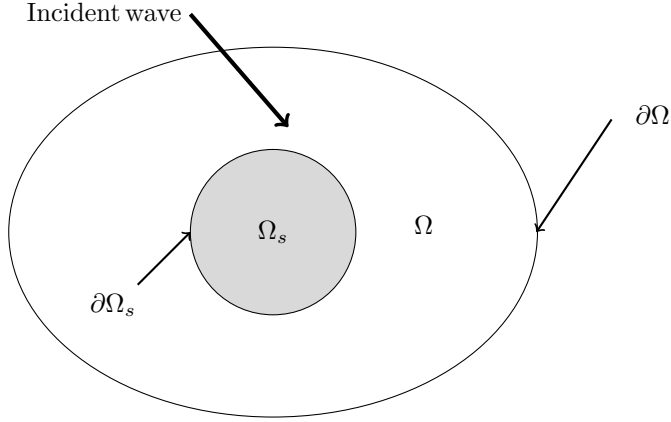\label{fig:1-1}

\section{Existence and uniqueness of the solutions}\label{sec-3}
In this section, we study the well-posedness of the problem (\ref{eq:1-9})-(\ref{eq:1-10}). To begin with, we introduce some notations. We denote $H^{1}(\Omega)$ as the conventional Sobolev spaces of complex-valued functions defined in $\Omega$ and $H^{1}_{0}(\Omega)$ as the subspace of $H^{1}(\Omega)$ consisting of functions whose traces are zero on $\partial \Omega$. Let ${L}^{p}(\Omega)$ and $\mathbf{L}^{p}(\Omega) = [L^{p}(\Omega)]^{3} $ be the Lebesgue spaces of complex-valued functions and vector-valued functions with 3 components, respectively. $ L^{2}$ inner-products in $L^{2}(\Omega) $ and $\mathbf{L}^{2}(\Omega)$ are denoted by $(\cdot,\cdot )$ without ambiguity. To avoid confusion, we use $(\cdot,\cdot )_{s}$ to denote the $ L^{2}$ inner-products in $L^{2}(\Omega_{s}) $ and $\mathbf{L}^{2}(\Omega_{s})$.

We define
\begin{equation}\label{eq:2-1}
\begin{array}{lll}
{\displaystyle \mathbf{H}(\mathbf{curl};\Omega)=\{\mathbf{u}\in
\mathbf{L}^{2}(\Omega)\,|\, \,\nabla\times \mathbf{u}\in \mathbf{L}^{2}(\Omega)
\}, }\\[2mm]
{\displaystyle \mathbf{H}(\mathbf{div};\Omega)=\{\mathbf{u}\in
\mathbf{L}^{2}(\Omega)\,|\,\, \nabla\cdot \mathbf{u}\in L^{2}(\Omega)
\}, }\\[2mm]
\end{array}
\end{equation}
which are equipped with the norms
\begin{equation*}
\begin{array}{lll}
{\displaystyle \|\mathbf{u}\|_{\mathbf{H}(\mathbf{curl};\Omega)}=\|\mathbf{u}\|_{\mathbf{L}^2(\Omega)}+\|\nabla\times \mathbf{u}\|_{\mathbf{L}^2(\Omega)}, }\\[2mm]
{\displaystyle \|\mathbf{u}\|_{\mathbf{H}(\mathbf{div};\Omega)}=\|\mathbf{u}\|_{\mathbf{L}^2(\Omega)}+\|\nabla\cdot \mathbf{u}\|_{L^2(\Omega)}. }
\end{array}
\end{equation*}
In addition,
\begin{equation}\label{eq:2-2}
\begin{array}{lll}
{\displaystyle \mathbf{H}_{T}(\mathbf{curl};\Omega)=\{\mathbf{u}\in
\mathbf{H}(\mathbf{curl};\Omega)\, |  \,\, \mathbf{u}_{T} = ({\bf n}\times {\bf u})\times{\bf n} \in {\bf L}^{2}(\partial \Omega)  \,\, \,{\rm on}\, \,\partial\Omega\}, }\\[2mm]
{\displaystyle \mathbf{H}_{0}(\mathbf{curl};\Omega)=\{\mathbf{u}\in
\mathbf{H}(\mathbf{curl};\Omega)\, |  \,\, \mathbf{u}\times \mathbf{n}={\bf 0} \,\, \,{\rm on}\, \,\partial\Omega\}, }\\[2mm]
{\displaystyle \mathbf{H}_{0}(\mathbf{div};\Omega)=\{\mathbf{u}\in
\mathbf{H}(\mathbf{div};\Omega)\, |  \,\, \mathbf{u}\cdot\mathbf{n}=0 \,\, \,{\rm on}\, \,\partial\Omega\}. }\\[2mm]
\end{array}
\end{equation}
Functions in $\mathbf{H}_{T}(\mathbf{curl};\Omega)$ are equipped with the norm
\begin{equation}
\|\mathbf{u}\|_{\mathbf{H}_{T}(\mathbf{curl};\Omega)}=\|\mathbf{u}\|_{\mathbf{L}^2(\Omega)}+\|\nabla\times \mathbf{u}\|_{\mathbf{L}^2(\Omega)} + \|\mathbf{u}_{T}\|_{\mathbf{L}^2(\partial \Omega)}.
\end{equation}
For the sake of convenience, we denote by
\begin{equation}\label{eq:2-3}
\begin{array}{lll}
{\displaystyle X(\Omega) =\mathbf{H}_{T}(\mathbf{curl};\Omega),\quad \|\mathbf{u}\|_{X(\Omega)}= \|\mathbf{u}\|_{\mathbf{H}(\mathbf{curl};\Omega)} , }\\[2mm]
{\displaystyle Y(\Omega_{s}) = \mathbf{H}_{0}(\mathbf{div};\Omega_{s}), \quad \|\mathbf{v}\|_{Y(\Omega_{s})} = \|\mathbf{v}\|_{\mathbf{H}(\mathbf{div};\Omega_{s})} }.
\end{array}
\end{equation}

We now give the weak formulation for the problem (\ref{eq:1-9})-(\ref{eq:1-10}). Given ${\bf g} \in {\bf L}^{2}(\partial \Omega)$, find $(\mathbf{E},{\bf J}) \in X(\Omega)\times Y(\Omega_{s})$, such that the equations 
 \begin{equation}\label{eq:2-4}
\left\{
\begin{array}{@{}l@{}}
{\displaystyle  \big(\mu^{-1}\nabla\times{\bf E}, \,\nabla\times{\bf u}) - \omega^{2}(\varepsilon{\bf E}, \,{\bf u}) - {\rm i}\omega \langle {\bf E}_{T}, {\bf u}_{T}\rangle = \langle {\bf g}, {\bf u}_{T}\rangle + {\rm i}\omega\big(\overline{\bf J}, \,{\bf u}\big) ,}\\[2mm]
 {\displaystyle \beta^{2}\big(\nabla\cdot {\bf J},\,\nabla\cdot {\bf v}\big)_{s}-\omega(\omega+{\rm i}\gamma)({\bf J},\,{\bf v})_{s} = -{\rm i}\omega \omega^{2}_{p}\varepsilon_{0}\big({\bf E}|_{\Omega_{s}},\,{\bf v}\big)_{s}} 
\end{array}
\right.
\end{equation}
hold for each $({\bf u}, {\bf v}) \in X(\Omega)\times Y(\Omega_{s})$, where $\varepsilon$ and $\mu$ are given in (\ref{eq:1-11}) and ${\bf u}_{T} = ({\bf n}\times {\bf u})\times {\bf n}$ with ${\bf n}$ being the unit outward normal to $\partial \Omega$. $\langle\cdot,\cdot\rangle$ denotes the $L^{2}$ inner product in $\mathbf{L}^{2}(\partial \Omega)$. $\overline{\bf J}$ and ${\bf E}|_{\Omega_{s}}$ are defined as follows.
\begin{equation}\label{eq:2-5}
\overline{\bf J}=\left\{
\begin{array}{l}
 {\bf J}, \,  \quad \hbox{in $\Omega_{s}$,}\\
 0, \,  \quad \hbox{in $\Omega/\Omega_{s}$,}
\end{array}
\right.
\quad {\bf E}|_{\Omega_{s}} = 1_{\Omega_{s}} {\bf E}\,,
\end{equation}
where $1_{\Omega_{s}}$ is the characteristic function of $\Omega_{s}$.

We now state the main result of this section.
\begin{theorem}\label{thm:2-1}
Let $\Omega$ and $\Omega_{s}$ be the bounded, simply-connected, Lipschitz polyhedron domains in $\mathbb{R}^{3}$ with $\bar{\Omega}_{s}\subset \Omega$. The equations (\ref{eq:2-4}) exist a unique solution $({\bf E},\,{\bf J}) \in X(\Omega)\times Y(\Omega_{s})$ satisfying
\begin{equation}\label{eq:2-5-0}
\Vert {\bf E} \Vert_{X(\Omega)} + \Vert {\bf J} \Vert_{Y(\Omega_{s})} \leq C \Vert {\bf g}\Vert_{{\bf L}^{2}(\partial \Omega)},
\end{equation}
where the constant $C$ might depend on $\mu_{1}$, $\varepsilon_{1}$,  $\mu_{2}$, $\varepsilon_{2}$, $\gamma$, $\beta$, $\omega$ and $\omega_{p}$.
\end{theorem}

We first prove the uniqueness of solutions of (\ref{eq:2-4}). 
\begin{lemma}\label{lem:2-1}
There is at most one solution $({\bf E},\,{\bf J})\in X(\Omega)\times Y(\Omega_{s})$ of (\ref{eq:2-4}).
\end{lemma}
\begin{proof}
Since (\ref{eq:2-4}) is a linear system, we only need to show that $({\bf E},\,{\bf J}) = ({\bf 0}, \bf{0})$ is the only solution of (\ref{eq:2-4}) with ${\bf g} = {\bf 0}$. To this end, we first choose ${\bf u} = {\bf E}$ in the first equation of (\ref{eq:2-4}) and take the imaginary part of the equation to obtain
\begin{equation}\label{eq:2-6}
-\Vert {\bf E}_{T}\Vert^{2}_{{\bf L}^{2}(\partial \Omega)} = {\rm Re}\big(\overline{\bf J}, \,{\bf E}\big).
\end{equation}
Next by setting ${\bf v} = {\bf J}$ in the second equation of (\ref{eq:2-4}) and taking the imaginary part of the equation, we have
\begin{equation}\label{eq:2-7}
\frac{\gamma}{\omega_{p}^{2}\varepsilon_{0}} \Vert {\bf J} \Vert^{2}_{{\bf L}^{2}(\Omega_{s})} = {\rm Re}\big({\bf E}|_{\Omega_{s}},\,{\bf J}\big)_{s}.
\end{equation}
Since $\big(\overline{\bf J}, \,{\bf E}\big)= \big({\bf E}|_{\Omega_{s}},\,{\bf J}\big)_{s}$, from (\ref{eq:2-6}) and (\ref{eq:2-7}), we deduce that
\begin{equation}
{\bf J} = {\bf 0}, \quad {\bf E}_{T} = {\bf 0}.
\end{equation}
Thus we find that ${\bf E}$ satisfies
\begin{equation}\label{eq:2-7-0}
 \big(\mu^{-1}\nabla\times{\bf E}, \,\nabla\times{\bf u}) - \omega^{2}(\varepsilon{\bf E}, \,{\bf u}) - {\rm i}\omega \langle {\bf E}_{T}, {\bf u}_{T}\rangle =0,\quad {\rm for}\,\,{\rm all}\,\,{\bf u}\in X(\Omega).
\end{equation}
It was proved in Theorem~4.12 of \cite{Monk} that the homogeneous problem (\ref{eq:2-7-0}) exists the only solution ${\bf E} = 0$. Consequently, $({\bf E},\,{\bf J}) = ({\bf 0}, \bf{0})$ is the only solution of (\ref{eq:2-5}) with ${\bf g} =0$.\qquad  \end{proof}

Before proving the existence of solutions of (\ref{eq:2-4}), we give two useful lemmas.
\begin{lemma}\label{lem:2-2}
We have the following Helmholtz decompositions for $X(\Omega)$ and $Y(\Omega_{s})$
\begin{equation}\label{eq:2-8}
X(\Omega) = X_{0}(\Omega) \oplus \nabla H_{0}^{1}(\Omega),\quad Y(\Omega_{s}) = Y_{0}(\Omega_{s})\oplus \nabla \times \mathbf{H}_{0}(\mathbf{curl};\Omega_{s}),
\end{equation}
where
\begin{equation}\label{eq:2-9}
\begin{array}{lll}
{\displaystyle X_{0}(\Omega) = \{ {\bf u}\in X(\Omega) \;\,| \;\,(\varepsilon {\bf u},\,\nabla \xi)= 0\quad {\rm for}\,\,{\rm all}\,\,\xi\in H_{0}^{1}(\Omega)\}, }\\[2mm]
{\displaystyle Y_{0}(\Omega_{s}) = \{ {\bf v}\in Y(\Omega_{s}) \;\,| \;\,( {\bf v},\,\nabla \times {\bf w})_{s}= 0\quad {\rm for}\,\,{\rm all}\,\,{\bf w}\in \mathbf{H}_{0}(\mathbf{curl};\Omega_{s})\}. }
\end{array}
\end{equation}
\end{lemma}
\begin{proof}
The Helmholtz decomposition for $X(\Omega)$ was proved in Lemma~4.5 of \cite{Monk}. It is not difficult to show that $\nabla \times \mathbf{H}_{0}(\mathbf{curl};\Omega_{s})$ is a closed subspace of $Y(\Omega_{s})$. Therefore the Helmholtz decomposition for $Y(\Omega_{s})$ follows from the projection theorem. For every ${\bf v} \in  Y(\Omega_{s})$, we can write it as
\begin{equation}
{\bf v} = {\bf v}_{0}+\nabla\times{\bf A},
\end{equation}
where ${\bf v}_{0} \in Y_{0}(\Omega_{s})$, ${\bf A}\in\mathbf{H}_{0}(\mathbf{curl};\Omega_{s})$. In particular, we can take ${\bf A}$ to be divergence-free, i.e.,$\nabla\cdot{\bf A} = 0$. \quad \end{proof}

\begin{lemma}\label{lem:2-3}
$X_{0}(\Omega)$ and $Y_{0}(\Omega_{s})$ are compactly embedded  in ${\bf L}^{2}(\Omega)$ and ${\bf L}^{2}(\Omega_{s})$, respectively.
\end{lemma}

The compact embedding of $X_{0}(\Omega)$ was proved in Theorem~4.7 of \cite{Monk} and the compactness property of $Y_{0}(\Omega_{s})$ can be proved by a similar trick. 

Now that we know $X(\Omega) = X_{0}(\Omega) \oplus \nabla H_{0}^{1}(\Omega)$ and $Y(\Omega_{s}) = Y_{0}(\Omega_{s})\oplus \nabla \times \mathbf{H}_{0}(\mathbf{curl};\Omega_{s})$, we can write any solution $({\bf E},\,{\bf J})$ of (\ref{eq:2-4}) as 
\begin{equation}\label{eq:2-10}
{\bf E} = {\bf E}_{0} + \nabla \varphi, \quad {\bf J} = {\bf J}_{0} + \nabla\times{\bf A}
\end{equation}
for some $({\bf E}_{0},\,{\bf J}_{0})\in X_{0}(\Omega) \times Y_{0}(\Omega_{s})$ and $(\varphi,\,{\bf A})\in H_{0}^{1}(\Omega) \times \mathbf{H}_{0}(\mathbf{curl};\Omega_{s})$. In addition, we assume that ${\bf A}$ satisfies $\nabla\cdot {\bf A} = 0$.

Substituting (\ref{eq:2-10}) into (\ref{eq:2-4}), we find that 
 \begin{equation}\label{eq:2-11}
\left\{
\begin{array}{@{}l@{}}
{\displaystyle  \big(\mu^{-1}\nabla\times{\bf E}_{0}, \,\nabla\times{\bf u}) - \omega^{2}(\varepsilon({\bf E}_{0}+\nabla \varphi), \,{\bf u}) - {\rm i}\omega \langle {\bf E}_{0,T}, {\bf u}_{T}\rangle = \langle {\bf g}, {\bf u}_{T}\rangle }\\[2mm]
{\displaystyle \qquad \qquad \qquad \qquad  \qquad \qquad \qquad \qquad \qquad \qquad \qquad+\, {\rm i}\omega\big(\overline{{\bf J}_{0} + \nabla\times{\bf A}}, \,{\bf u}\big) ,}\\[2mm]
 {\displaystyle -\omega(\omega+{\rm i}\gamma)(({\bf J}_{0} + \nabla\times{\bf A}),\,{\bf v})_{s} + \beta^{2}\big(\nabla\cdot {\bf J}_{0},\,\nabla\cdot {\bf v}\big)_{s} = -{\rm i}\omega \omega^{2}_{p}\varepsilon_{0}\big(({\bf E}_{0} + \nabla \varphi)|_{\Omega_{s}},\,{\bf v}\big)_{s}}
\end{array}
\right.
\end{equation}
for all $({\bf u}, {\bf v}) \in X(\Omega)\times Y(\Omega_{s})$. 

Now taking $({\bf u}, {\bf v}) = (\nabla \xi, \,\nabla \times{\bf w})$ in (\ref{eq:2-11}), where $ \xi \in H_{0}^{1}(\Omega)$ and ${\bf w} \in\mathbf{H}_{0}(\mathbf{curl};\Omega_{s})$, we obtain
\begin{equation}\label{eq:2-12-0}
-\omega^{2}\big(\varepsilon \nabla \varphi,\,\nabla\xi\big) = {\rm i}\omega\big(\overline{\bf J}_{0}, \,\nabla\xi), \quad {\rm for}\,\,{\rm all}\,\,\xi \in H_{0}^{1}(\Omega),
\end{equation}
and
\begin{equation}\label{eq:2-12}
\left\{
\begin{array}{@{}l@{}}
{\displaystyle -\omega(\omega+{\rm i}\gamma)\big( \nabla\times{\bf A},\, \nabla \times{\bf w}\big)_{s} =  -{\rm i}\omega \omega^{2}_{p}\varepsilon_{0}\big({\bf E}_{0} |_{\Omega_{s}},\,\nabla \times{\bf w}\big)_{s}, \;\; {\rm for}\,\,{\rm all}\,\,{\bf w}\in \mathbf{H}_{0}(\mathbf{curl};\Omega_{s}),}\\[2mm]
{\displaystyle \big({\bf A},\, \nabla \eta)_{s} = 0, \quad {\rm for}\,\,{\rm all}\,\, \eta \in H_{0}^{1}(\Omega_{s}),}
\end{array}
\right.
\end{equation}
where we have used the fact that $\big(\overline{\nabla\times{\bf A}}, \, \nabla \xi\big) =  \big(\nabla\times \overline{\bf A}, \, \nabla \xi\big) =0 $ and $\big((\nabla \varphi)|_{\Omega_{s}},\,\nabla\times {\bf w}\big)_{s} = 0$. By introducing a Lagrangian multiplier $q\in H_{0}^{1}(\Omega_s)$, we can rewrite (\ref{eq:2-12}) as
\begin{equation}\label{eq:2-13-0}
\left\{
\begin{array}{@{}l@{}}
{\displaystyle \omega(\omega+{\rm i}\gamma)\big( \nabla\times{\bf A},\, \nabla \times{\bf w}\big)_{s} + \big(\nabla q, \, {\bf w}\big)_{s}=  {\rm i}\omega \omega^{2}_{p}\varepsilon_{0}\big({\bf E}_{0} |_{\Omega_{s}},\,\nabla \times{\bf w}\big)_{s},}\\[2mm]
{\displaystyle \qquad \qquad \qquad \qquad  \qquad \qquad \qquad \qquad \qquad \qquad \qquad \; {\rm for}\,\,{\rm all}\,\,{\bf w}\in \mathbf{H}_{0}(\mathbf{curl};\Omega_{s}),}\\[0.5mm]
{\displaystyle \big({\bf A},\, \nabla \eta)_{s} = 0, \quad {\rm for}\,\,{\rm all}\,\, \eta \in H_{0}^{1}(\Omega_{s}).}
\end{array}
\right.
\end{equation}
Next we take $({\bf u}, {\bf v}) \in  X_{0}(\Omega) \times Y_{0}(\Omega_{s})$ in (\ref{eq:2-11}) to obtain
\begin{equation}\label{eq:2-13}
\left\{
\begin{array}{@{}l@{}}
{\displaystyle  \big(\mu^{-1}\nabla\times{\bf E}_{0}, \,\nabla\times{\bf u}\big) - \omega^{2}(\varepsilon{\bf E}_{0}, \,{\bf u}) - {\rm i}\omega \langle {\bf E}_{0,T}, {\bf u}_{T}\rangle = \langle {\bf g}, {\bf u}_{T}\rangle + {\rm i}\omega\big(\overline{{\bf J}_{0} + \nabla\times{\bf A}}, \,{\bf u}\big) ,}\\[2mm]
 {\displaystyle -\omega(\omega+{\rm i}\gamma)\big({\bf J}_{0} ,\,{\bf v}\big)_{s} + \beta^{2}\big(\nabla\cdot {\bf J}_{0},\,\nabla\cdot {\bf v}\big)_{s} = -{\rm i}\omega \omega^{2}_{p}\varepsilon_{0}\big(({\bf E}_{0} + \nabla \varphi)|_{\Omega_{s}},\,{\bf v}\big)_{s}.} 
\end{array}
\right.
\end{equation}

Combining (\ref{eq:2-11})-(\ref{eq:2-13}), we now reformulate the problem (\ref{eq:2-4}) as follows. Find $({\bf E}_{0},\,{\bf J}_{0})\in X_{0}(\Omega) \times Y_{0}(\Omega_{s})$, such that
\begin{equation}\label{eq:2-14}
\left\{
\begin{array}{@{}l@{}}
{\displaystyle  \big(\mu^{-1}\nabla\times{\bf E}_{0}, \,\nabla\times{\bf u}\big) - \omega^{2}(\varepsilon{\bf E}_{0}, \,{\bf u}) - {\rm i}\omega \langle {\bf E}_{0,T}, {\bf u}_{T}\rangle + \beta^{2}\big(\nabla\cdot {\bf J}_{0},\,\nabla\cdot {\bf v}\big)_{s}}\\[2mm]
{\displaystyle \; -\,\omega(\omega+{\rm i}\gamma)\big({\bf J}_{0} ,\,{\bf v}\big)_{s}  +{\rm i}\omega \omega^{2}_{p}\varepsilon_{0}\big({\bf E}_{0}|_{\Omega_{s}},\,{\bf v}\big)_{s} - {\rm i}\omega\big(\overline{\bf J}_{0}, \,{\bf u}\big) }\\[2mm]
{\displaystyle \; =\, {\rm i}\omega\big(\overline{\nabla\times{\bf A}}, \,{\bf u}\big) -{\rm i}\omega \omega^{2}_{p}\varepsilon_{0}\big((\nabla \varphi)|_{\Omega_{s}},\,{\bf v}\big)_{s} + \langle {\bf g}, {\bf u}_{T}\rangle} 
\end{array}
\right.
\end{equation}
for all $({\bf u}, {\bf v}) \in  X_{0}(\Omega) \times Y_{0}(\Omega_{s})$, where $\varphi \in H_{0}^{1}(\Omega)$ and ${\bf A} \in \mathbf{H}_{0}(\mathbf{curl};\Omega_{s})$ satisfy (\ref{eq:2-12-0}) and (\ref{eq:2-13-0}), respectively.

We now define the sesquilinear form $a_{+}$ by
\begin{equation}\label{eq:2-15}
\begin{array}{lll}
{\displaystyle a_{+}(\hat{{\bf u}}, \,\hat{{\bf v}}) = \big(\mu^{-1}\nabla\times{\bf u}_{1}, \,\nabla\times{\bf v}_{1}\big) + \omega^{2}(\varepsilon{\bf u}_{1}, \,{\bf v}_{1}) - {\rm i}\omega \langle {\bf u}_{1,T}, {\bf v}_{1,T}\rangle }\\[2mm]
{\displaystyle \qquad \qquad \quad +\, \beta^{2}\big(\nabla\cdot {\bf u}_{2},\,\nabla\cdot {\bf v}_{2}\big)_{s}+\omega(\omega-{\rm i}\gamma)\big({\bf u}_{2} ,\,{\bf v}_{2}\big)_{s} }
\end{array}
\end{equation}
for all $\hat{{\bf u}} = ({\bf u}_{1},{\bf u}_{2}),  \hat{{\bf v}} = ({\bf v}_{1},{\bf v}_{2})\in X_{0}(\Omega) \times Y_{0}(\Omega_{s})$. For convenience, we introduce the following notations. 
\begin{equation*}
\Vert \hat{{\bf u}} \Vert_{X(\Omega)\times Y(\Omega_{s})}:=\Vert {\bf u}_{1} \Vert_{X(\Omega)} + \Vert {\bf u}_{2} \Vert_{Y(\Omega_{s})},\; \Vert \hat{{\bf u}} \Vert_{{\bf L}^{2}(\Omega)\times {\bf L}^{2}(\Omega_{s})}:=\Vert {\bf u}_{1} \Vert_{{\bf L}^{2}(\Omega)} + \Vert {\bf u}_{2} \Vert_{{\bf L}^{2}(\Omega_{s})}.
\end{equation*}
It is not difficult to prove the following lemma.
\begin{lemma}\label{lem:2-4}
There exists a constant $\alpha>0$ depending on $\mu_{1}$, $\varepsilon_{1}$,  $\mu_{2}$, $\varepsilon_{2}$, $\gamma$, $\beta$, $\omega$ and $\omega_{p}$ such that
\begin{equation}\label{eq:2-16}
|a_{+}(\hat{{\bf u}}, \,\hat{{\bf u}})|\geq \alpha \Vert \hat{{\bf u}} \Vert_{X(\Omega)\times Y(\Omega_{s})}, \quad {\rm for}\,\, {\rm all}\,\,\hat{{\bf u}} = ({\bf u}_{1},{\bf u}_{2}) \in X_{0}(\Omega) \times Y_{0}(\Omega_{s}).
\end{equation}
\end{lemma}
To proceed further, we define a map $K$: ${\bf L}^{2}(\Omega)\times{\bf L}^{2}(\Omega_{s})\rightarrow {\bf L}^{2}(\Omega)\times{\bf L}^{2}(\Omega_{s})$ such that if $\hat{{\bf u}} = ({\bf u}_{1},{\bf u}_{2}) \in {\bf L}^{2}(\Omega)\times{\bf L}^{2}(\Omega_{s})$ then $K\hat{{\bf u}} \in X_{0}(\Omega) \times Y_{0}(\Omega_{s})\subset {\bf L}^{2}(\Omega)\times{\bf L}^{2}(\Omega_{s})$ satisfies
\begin{equation}\label{eq:2-17}
\begin{array}{lll}
{\displaystyle  a_{+}(K\hat{{\bf u}}, \,\hat{{\bf v}}) = -2\omega^{2}(\varepsilon{\bf u}_{1}, \,{\bf v}_{1}) -2\omega^{2}\big({\bf u}_{2} ,\,{\bf v}_{2}\big)_{s}-{\rm i}\omega \big(\overline{\bf u}_{2}, \,{\bf v}_{1}\big)+{\rm i}\omega \omega^{2}_{p}\varepsilon_{0}\big({\bf u}_{1}|_{\Omega_{s}},\,{\bf v}_{2}\big)_{s} }\\[2mm]
{\displaystyle \; - \,{\rm i}\omega\big(\overline{\nabla\times{\bf A}}, \,{\bf v}_{1}\big) + {\rm i}\omega \omega^{2}_{p}\varepsilon_{0}\big((\nabla \varphi)|_{\Omega_{s}},\,{\bf v}_{2}\big)_{s}, \; {\rm for}\,\, {\rm all}\,\,\hat{{\bf v}} = ({\bf v}_{1},{\bf v}_{2}) \in X_{0}(\Omega) \times Y_{0}(\Omega_{s}),}
\end{array}
\end{equation}
where $\varphi\in H_{0}^{1}(\Omega)$ and ${\bf A} \in \mathbf{H}_{0}(\mathbf{curl};\Omega_{s})$ are the solutions of 
\begin{equation}\label{eq:2-18}
-\omega^{2}\big(\varepsilon \nabla \varphi,\,\nabla\xi\big) = {\rm i}\omega\big(\overline{\bf u}_{2}, \,\nabla\xi), \quad {\rm for}\,\,{\rm all}\,\,\xi \in H_{0}^{1}(\Omega),
\end{equation}
and 
\begin{equation}\label{eq:2-19}
\left\{
\begin{array}{@{}l@{}}
{\displaystyle \omega(\omega+{\rm i}\gamma)\big( \nabla\times{\bf A},\, \nabla \times{\bf w}\big)_{s} + \big(\nabla q, \, {\bf w}\big)_{s}=  {\rm i}\omega \omega^{2}_{p}\varepsilon_{0}\big({\bf u}_{1} |_{\Omega_{s}},\,\nabla \times{\bf w}\big)_{s}, }\\[2mm]
{\displaystyle \qquad \qquad \qquad \qquad  \qquad \qquad \qquad \qquad \qquad \qquad \qquad\; {\rm for}\,\,{\rm all}\,\,{\bf w}\in \mathbf{H}_{0}(\mathbf{curl};\Omega_{s}),}\\[0.5mm]
{\displaystyle \big({\bf A},\, \nabla \eta)_{s} = 0, \quad {\rm for}\,\,{\rm all}\,\, \eta \in H_{0}^{1}(\Omega_{s}).}
\end{array}
\right.
\end{equation}

We have the following result.
\begin{theorem}\label{thm:2-2}
The operator $K$ is a bounded and compact map from ${\bf L}^{2}(\Omega)\times{\bf L}^{2}(\Omega_{s})$ into ${\bf L}^{2}(\Omega)\times{\bf L}^{2}(\Omega_{s})$. Moreover,
\begin{equation}\label{eq:2-20}
\Vert K \hat{{\bf u}} \Vert_{X(\Omega)\times Y(\Omega_{s})} \leq C \Vert \hat{{\bf u}} \Vert_{{\bf L}^{2}(\Omega)\times {\bf L}^{2}(\Omega_{s})}.
\end{equation}
\end{theorem}
\begin{proof}
This theorem can be proved by using the Lax--Milgram theorem. To begin with, we check the conditions of the Lax--Milgram theorem. It is not difficult to show that $a_{+}(\cdot,\,\cdot)$ is bounded. That is, there exists a constant $C$, such that 
\begin{equation}
|a_{+}(\hat{{\bf u}}, \,\hat{{\bf v}})|\leq C \Vert \hat{{\bf u}} \Vert_{X(\Omega)\times Y(\Omega_{s})} \Vert \hat{{\bf v}} \Vert_{X(\Omega)\times Y(\Omega_{s})}
\end{equation}
holds for all $\hat{{\bf u}},\, \hat{{\bf v}} \in X_{0}(\Omega) \times Y_{0}(\Omega_{s})$. Coercivity of $a_{+}(\cdot,\,\cdot)$ has been given in Lemma~\ref{lem:2-4}. It remains to show that there exists a constant $C$ such that
\begin{equation}\label{eq:2-21}
\begin{array}{lll}
{\displaystyle |-2\omega^{2}(\varepsilon{\bf u}_{1}, \,{\bf v}_{1}) -2\omega^{2}\big({\bf u}_{2} ,\,{\bf v}_{2}\big)_{s}-{\rm i}\omega \big(\overline{\bf u}_{2}, \,{\bf v}_{1}\big)+{\rm i}\omega \omega^{2}_{p}\varepsilon_{0}\big({\bf u}_{1}|_{\Omega_{s}},\,{\bf v}_{2}\big)_{s}  }\\[2mm]
{\displaystyle -\, {\rm i}\omega\big(\overline{\nabla\times{\bf A}}, \,{\bf v}_{1}\big) + {\rm i}\omega \omega^{2}_{p}\varepsilon_{0}\big((\nabla \varphi)|_{\Omega_{s}},\,{\bf v}_{2}\big)_{s} |\leq C \Vert \hat{{\bf u}} \Vert_{{\bf L}^{2}(\Omega)\times {\bf L}^{2}(\Omega_{s})}  \Vert \hat{{\bf v}} \Vert_{{\bf L}^{2}(\Omega)\times {\bf L}^{2}(\Omega_{s})},}
\end{array}
\end{equation}
where $\varphi \in H_{0}^{1}(\Omega)$ and ${\bf A} \in \mathbf{H}_{0}(\mathbf{curl};\Omega_{s})$ are the solutions of (\ref{eq:2-18}) and (\ref{eq:2-19}), respectively. To prove (\ref{eq:2-21}), it suffices to deduce that
\begin{equation}\label{eq:2-22}
|\big((\nabla \varphi)|_{\Omega_{s}},\,{\bf v}_{2}\big)_{s}| + |\big(\overline{\nabla\times{\bf A}}, \,{\bf v}_{1}\big)|\leq C \Vert \hat{{\bf u}} \Vert_{{\bf L}^{2}(\Omega)\times {\bf L}^{2}(\Omega_{s})}  \Vert \hat{{\bf v}} \Vert_{{\bf L}^{2}(\Omega)\times {\bf L}^{2}(\Omega_{s})}.
\end{equation}
Since $\varphi$ satisfies (\ref{eq:2-18}), it is easy to see that 
\begin{equation}\label{eq:2-23}
\Vert \nabla \varphi\Vert _{{\bf L}^{2}(\Omega)} \leq C \Vert {\bf u}_{2}\Vert_{{\bf L}^{2}(\Omega_{s})}.
\end{equation}
By using the classical theory of variational problems \cite{Monk}, we see that the mixed problem (\ref{eq:2-19}) exists a unique solution $({\bf A},q)$. In addition, 
\begin{equation}\label{eq:2-24}
\Vert \nabla \times {\bf A}\Vert_{{\bf L}^{2}(\Omega_{s})}\leq C \Vert {\bf u}_{1}\Vert_{{\bf L}^{2}(\Omega)}.
\end{equation}
Combining (\ref{eq:2-23}) and (\ref{eq:2-24}), we have (\ref{eq:2-22}), which yields (\ref{eq:2-21}). Having verified the conditions of the Lax--Milgram theorem, we know $K\hat{{\bf u}}$ is well defined and obtain (\ref{eq:2-20}). The compactness of $K$ can be proved by applying a similar argument in Theorem~4.11 of \cite{Monk} and we omit the proof here.  \quad \quad \end{proof}

Next we define a vector $\hat{{\bf F}} = ({\bf F}_{1}, {\bf F}_{2}) \in  X_{0}(\Omega) \times Y_{0}(\Omega_{s})$
which satisfies
\begin{equation}\label{eq:2-25}
a_{+}(\hat{{\bf F}}, \,\hat{{\bf v}}) =  \langle {\bf g}, {\bf v}_{1,T}\rangle, \quad {\rm for}\,\, {\rm all}\,\,\hat{{\bf v}} = ({\bf v}_{1},{\bf v}_{2}) \in X_{0}(\Omega) \times Y_{0}(\Omega_{s}).
\end{equation}
By using the Lax--Milgram theorem again, we see that $\hat{{\bf F}}$ is well defined and 
\begin{equation}\label{eq:2-26}
\Vert \hat{{\bf F}} \Vert_{X(\Omega)\times Y(\Omega_{s})} \leq C \Vert {\bf g}\Vert_{{\bf L}^{2}(\partial \Omega)}.
\end{equation}
By virtue of the operator $K$, we find that the problem (\ref{eq:2-14}) is equivalent to finding $\hat{{\bf Z}} = (\mathbf{E}_{0},\mathbf{J}_{0}) \in X_{0}(\Omega) \times Y_{0}(\Omega_{s})$ such that 
\begin{equation}\label{eq:2-27}
(I+K)\hat{{\bf Z}} = \hat{{\bf F}}.
\end{equation}
Since $K$ is compact, by applying the Fredholm alternative theorem and Lemma~\ref{lem:2-1}, we see that (\ref{eq:2-27}) exists a unique solution $\hat{{\bf Z}}$ with the following estimate
\begin{equation}\label{eq:2-28}
\Vert \hat{{\bf Z}} \Vert_{{\bf L}^{2}(\Omega)\times {\bf L}^{2}(\Omega_{s})}\leq C \Vert \hat{{\bf F}} \Vert_{{\bf L}^{2}(\Omega)\times {\bf L}^{2}(\Omega_{s})}.
\end{equation}
Note that (\ref{eq:2-27}) implies that $\hat{{\bf Z}} = \hat{{\bf F}} - K\hat{{\bf Z}}$, from which we deduce
\begin{equation}\label{eq:2-29}
\begin{array}{lll}
{\displaystyle \Vert \hat{{\bf Z}} \Vert_{X(\Omega)\times Y(\Omega_{s})} \leq C\big(\Vert \hat{{\bf F}} \Vert_{X(\Omega)\times Y(\Omega_{s})} + \Vert K\hat{{\bf Z}} \Vert_{X(\Omega)\times Y(\Omega_{s})}\big)}\\[2mm]
{\displaystyle \qquad \qquad \qquad \quad \leq C\big(\Vert \hat{{\bf F}} \Vert_{X(\Omega)\times Y(\Omega_{s})} + \Vert \hat{{\bf Z}} \Vert_{{\bf L}^{2}(\Omega)\times {\bf L}^{2}(\Omega_{s})}\big),}
\end{array}
\end{equation}
where we have used (\ref{eq:2-20}).
Substituting (\ref{eq:2-28}) into (\ref{eq:2-29}) and applying (\ref{eq:2-26}), we arrive at 
\begin{equation}\label{eq:2-30}
\Vert {\bf E}_{0} \Vert_{X(\Omega)} + \Vert {\bf J}_{0}\Vert_{Y(\Omega_{s})}= \Vert \hat{{\bf Z}} \Vert_{X(\Omega)\times Y(\Omega_{s})} \leq C\Vert {\bf g}\Vert_{{\bf L}^{2}(\partial \Omega)}.
\end{equation}
Since $\varphi$ and ${\bf A}$ satisfy (\ref{eq:2-12-0}) and (\ref{eq:2-13-0}), respectively, it follow from (\ref{eq:2-30}) that
\begin{equation}\label{eq:2-31}
\Vert \nabla \varphi\Vert_{{\bf L}^{2}(\Omega)}+ \Vert \nabla \times {\bf A}\Vert_{{\bf L}^{2}(\Omega_{s})}\leq \Vert {\bf E}_{0} \Vert_{{\bf L}^{2}(\Omega)} + \Vert {\bf J}_{0}\Vert_{{\bf L}^{2}(\Omega_{s})} \leq C\Vert {\bf g}\Vert_{{\bf L}^{2}(\partial \Omega)}.
\end{equation}
Combining (\ref{eq:2-30})-(\ref{eq:2-31}) and recalling the Helmholtz decompositions for ${\bf E}$ and ${\bf J}$ in (\ref{eq:2-10}), we have (\ref{eq:2-5-0}) and complete the proof of Theorem~\ref{thm:2-1}.

\begin{corollary}
Under the assumptions of Theorem~\ref{thm:2-1}, there exists a $\delta >0$ such that for all $t$ with $0\leq t< \delta$, ${\bf J} \in {\bf H}^{\frac{1}{2}+t}(\Omega_{s})$. If in addition, $\Omega_{s}$ is a convex polyhedron, we have ${\bf J} \in {\bf H}^{1}(\Omega_{s})$. Moreover, if $\varepsilon $ is constant on the whole domain $\Omega$, then ${\bf E} \in {\bf H}^{\frac12}(\Omega)$.
\end{corollary}
\begin{proof}
By taking ${\bf v} = \nabla\times{\bf w}$ with ${\bf w}\in \mathbf{H}_{0}(\mathbf{curl};\Omega_{s})$ in the second equation of (\ref{eq:2-4}) and using (\ref{eq:2-5-0}), we get
\begin{equation}\label{eq:2-32}
|\big({\bf J},\,\nabla\times{\bf w}\big)_{s}|\leq C |\big(\nabla\times {\bf E}|_{\Omega_{s}},\,{\bf w}\big)_{s}| \leq C \Vert \nabla\times{\bf E}\Vert_{{\bf L}^{2}(\Omega)} \Vert {\bf w}\Vert_{{\bf L}^{2}(\Omega_{s})}\leq C\Vert {\bf w}\Vert_{{\bf L}^{2}(\Omega_{s})},
\end{equation}
which implies that ${\bf J}\in \mathbf{H}(\mathbf{curl};\Omega_{s})$ and thus ${\bf J}\in  \mathbf{H}_{0}(\mathbf{div};\Omega_{s}) \cap \mathbf{H}(\mathbf{curl};\Omega_{s})$. Applying Theorem~3.50 of \cite{Monk}, we know that there is a $\delta >0$ such that ${\bf J} \in {\bf H}^{\frac{1}{2}+t}(\Omega_{s})$ holds for all $t$ with $0\leq t< \delta$. If $\Omega_{s}$ is a convex polyhedron, then $\mathbf{H}_{0}(\mathbf{div};\Omega_{s}) \cap \mathbf{H}(\mathbf{curl};\Omega_{s})$ is continuously embedded into ${\bf H}^{1}(\Omega_{s})$ (Theorem~3.9 of \cite{Gir}). Therefore we have ${\bf J} \in {\bf H}^{1}(\Omega_{s})$.

If $\varepsilon $ is constant on the whole domain $\Omega$, selecting ${\bf u} = \nabla \xi$ with $\xi\in H_{0}^{1}(\Omega)$ in the first equation of (\ref{eq:2-4}) and applying a similar argument, we find
\begin{equation}
|\big({\bf E},\,\nabla \xi\big)| \leq C \Vert \xi\Vert_{L^{2}(\Omega)}.
\end{equation}
Consequently, we have ${\bf E} \in \mathbf{H}(\mathbf{div};\Omega)\cap \mathbf{H}_{T}(\mathbf{curl};\Omega)$, which yields that ${\bf E} \in {\bf H}^{\frac12}(\Omega)$ (Theorem~3.47 of \cite{Monk}). \qquad \end{proof}

\section{Finite element approximation}\label{sec-4} In this section, we present the finite element approximation for the system (\ref{eq:2-4}) and prove the
convergence of the scheme. Our proof relies on the theory of collective compact operators which has been used to prove the convergence of finite element approximations for Maxwell's equations in Chapter~4 of \cite{Monk}.

Let $\mathcal{T}_{h}$ be a quasiuniform triangulation of $\Omega$ into tetrahedrons of maximal diameter $h$ which matches with the interface $\partial \Omega_{s}$, i.e., both triangulations for $\Omega_{s}$ and $\Omega/\Omega_{s}$ are combined into a standard triangulation of the whole domain $\Omega$. For convenience, we denote by $\mathcal{T}_{h,s}$ the restriction of $\mathcal{T}_{h}$ in the domain $\Omega_{s}$. Let $P_{r}$ be the spaces of polynomials of maximal total degree $r$ and $\widetilde{P}_{r}$ be the spaces of homogeneous polynomials of total degree exactly $r$. We define the finite element space of $H_{0}^{1}(\Omega)$ 
\begin{equation}
S_{h} = \{u_{h} \in H_{0}^{1}(\Omega):\; u_{h}|_{K} \in P_{r}, \;\forall K \in \mathcal{T}_{h}\},
\end{equation}
and the N\'{e}d\'{e}lec $\mathbf{H}_{T}(\mathbf{curl};\Omega)$-conforming and Raviart--Thomas $ \mathbf{H}_{0}(\mathbf{div};\Omega_{s})$-conforming finite element spaces:
\begin{equation}
\begin{array}{lll}
{\displaystyle  {X}_{h} = \{{\bf u}_{h} \in \mathbf{H}_{T}(\mathbf{curl};\Omega): \; {\bf u}_{h}|_{K}  \in ({P}_{r-1})^{3} \oplus \mathcal{R}_{r}, \; \forall \; K \in  \mathcal{T}_{h}\}, }\\[2mm]
{\displaystyle {Y}_{h,s} = \{{\bf u}_{h} \in \mathbf{H}_{0}(\mathbf{div};\Omega_{s}): \; {\bf u}_{h}|_{K}  \in (P_{r-1})^{3}\oplus\widetilde{P}_{r-1}{\bf x}, \; \forall \; K \in  \mathcal{T}_{h,s}\}, }
\end{array}
\end{equation}
where $\mathcal{R}_{r}$ is a subspace of homogeneous vector polynomials of degree $r$
\begin{equation*}
 \mathcal{R}_{r} = \{{\bf u} \in (\widetilde{P}_{r})^{3}: \;{\bf x}\cdot {\bf u} =0 \}.
\end{equation*}
In addition, we use $Q_{h,s}$ and $S_{h,s}$ to denote the degree-$r$ $\mathbf{H}_{0}(\mathbf{curl};\Omega_{s})$-conforming and $H^{1}_{0}(\Omega_{s})$-conforming finite element spaces, respectively.

Given ${\bf g}\in \mathbf{L}^{2}(\partial \Omega)$, we seek to approximate the solution $({\bf E},{\bf J}) \in X(\Omega)\times Y(\Omega_{s})$ of (\ref{eq:2-4}) by finding $({\bf E}_{h},{\bf J}_{h}) \in X_{h}\times Y_{h,s}$ such that the system
\begin{equation}\label{eq:3-1}
\left\{
\begin{array}{@{}l@{}}
{\displaystyle  \big(\mu^{-1}\nabla\times{\bf E}_{h}, \,\nabla\times{\bf u}_{h}) - \omega^{2}(\varepsilon{\bf E}_{h}, \,{\bf u}_{h}) - {\rm i}\omega \langle {\bf E}_{h,T}, {\bf u}_{h,T}\rangle = \langle {\bf g}, {\bf u}_{h,T}\rangle + {\rm i}\omega\big(\overline{\bf J}_{h}, \,{\bf u}_{h}\big) ,}\\[2mm]
 {\displaystyle \beta^{2}\big(\nabla\cdot {\bf J}_{h},\,\nabla\cdot {\bf v}_{h}\big)_{s}-\omega(\omega+{\rm i}\gamma)({\bf J}_{h},\,{\bf v}_{h})_{s} = -{\rm i}\omega \omega^{2}_{p}\varepsilon_{0}\big({{\bf E}_{h}}|_{\Omega_{s}},\,{\bf v}_{h}\big)_{s}} 
\end{array}
\right.
\end{equation}
holds for all $({\bf u}_{h}, {\bf v}_{h}) \in X_{h}\times Y_{h,s}$.

We first show the solvability of the approximate system (\ref{eq:3-1}).
\begin{lemma}\label{lem3-1}
There exists an $h_{0}>0$, such that for $h<h_{0}$, the discrete system (\ref{eq:3-1}) exists a unique solution $({\bf E}_{h},{\bf J}_{h})$.
\end{lemma}
\begin{proof}
Since (\ref{eq:3-1}) is a finite-dimensional linear system, it suffices to prove uniqueness. Let ${\bf g} = 0$. By a similar argument as in Lemma~\ref{lem:2-1}, we have ${\bf J}_{h}=0$ and thus
\begin{equation}\label{eq:3-2}
 \big(\mu^{-1}\nabla\times{\bf E}_{h}, \,\nabla\times{\bf u}_{h}) - \omega^{2}(\varepsilon{\bf E}_{h}, \,{\bf u}_{h}) - {\rm i}\omega \langle {\bf E}_{h,T}, {\bf u}_{h,T}\rangle =0,\quad {\rm for}\,\,{\rm all}\,\,{\bf u}_{h}\in X_{h}.
\end{equation}
Since the homogeneous equation (\ref{eq:3-2}) exists a unique solution ${\bf E}_{h}=0$ if $h$ is sufficiently small (Chapter~4 of \cite{Monk}), uniqueness is proved. \qquad \end{proof}

We now give the discrete version of Lemma~\ref{lem:2-2}, i.e., the discrete Helmholtz decompositions.
\begin{lemma}\label{lem3-2}
The finite element spaces $X_{h}$ and $Y_{h,s}$ can be decomposed as 
\begin{equation}\label{eq:3-3}
X_{h} = X_{0,h} \oplus \nabla S_{h},\quad Y_{h,s} = Y_{0,h,s}\oplus \nabla \times Q_{h,s},
\end{equation}
where
\begin{equation}\label{eq:3-4}
\begin{array}{lll}
{\displaystyle X_{0,h} = \{ {\bf u}_{h}\in X_{h} \;\,| \;\,(\varepsilon {\bf u}_{h},\,\nabla \xi_{h})= 0\quad {\rm for}\,\,{\rm all}\,\,\xi_{h}\in S_{h}\}, }\\[2mm]
{\displaystyle Y_{0,h,s} = \{ {\bf v}_{h}\in Y_{h,s} \;\,| \;\,( {\bf v}_{h},\,\nabla \times {\bf w}_{h})_{s}= 0\quad {\rm for}\,\,{\rm all}\,\,{\bf w}_{h}\in Q_{h,s}\}. }
\end{array}
\end{equation}
\end{lemma}
Lemma~\ref{lem3-2} follows from the projection theorem and the fact that $\nabla S_{h}$ and $\nabla \times Q_{h,s}$ are respectively the subspaces of $X_{h}$ and $Y_{h,s}$.

In virtue of Lemma~\ref{lem3-2}, we can write the solution $({\bf E}_{h}, {\bf J}_{h})$ of (\ref{eq:3-1}) as 
\begin{equation}\label{eq:3-5}
{\bf E}_{h} = {\bf E}_{0,h} + \nabla \varphi_{h}, \quad {\bf J}_{h} = {\bf J}_{0,h} + \nabla\times{\bf A}_{h}
\end{equation}
for some $({\bf E}_{0,h},\,{\bf J}_{0,h})\in X_{0,h} \times Y_{0,h,s}$ and $(\varphi_{h},\,{\bf A}_{h})\in S_{h} \times Q_{h,s}$. In addition, we assume that ${\bf A}_{h}$ is discrete divergence-free, i.e.,
\begin{equation}
({\bf A}_{h},\nabla \eta_{h})_{s} = 0, \quad {\rm for}\,\,{\rm all}\,\,{\eta_{h}} \in S_{h,s}.
\end{equation}
Substituting (\ref{eq:3-5}) into (\ref{eq:3-1}) and selecting $({\bf u}_{h}, {\bf v}_{h}) = (\nabla \xi_{h}, \,\nabla \times{\bf w}_{h})$, where $ \xi_{h} \in S_{h}$ and ${\bf w} _{h}\in Q_{h,s}$, we obtain
\begin{equation}\label{eq:3-6}
-\omega^{2}\big(\varepsilon \nabla \varphi_{h},\,\nabla\xi_{h}\big) = {\rm i}\omega\big(\overline{\bf J}_{0,h}, \,\nabla\xi_{h}), \quad {\rm for}\,\,{\rm all}\,\,\xi_{h} \in S_{h},
\end{equation}
and
\begin{equation}\label{eq:3-7}
\left\{
\begin{array}{@{}l@{}}
{\displaystyle -\omega(\omega+{\rm i}\gamma)\big( \nabla\times{\bf A}_{h},\, \nabla \times{\bf w}_{h}\big)_{s} =  -{\rm i}\omega \omega^{2}_{p}\varepsilon_{0}\big({\bf E}_{0,h} |_{\Omega_{s}},\,\nabla \times{\bf w}_{h}\big)_{s}, \;\; {\rm for}\,\,{\rm all}\,\,{\bf w}_{h}\in Q_{h,s},}\\[2mm]
{\displaystyle \big({\bf A}_{h},\, \nabla \eta_{h})_{s} = 0, \quad {\rm for}\,\,{\rm all}\,\, \eta_{h} \in S_{h,s},}
\end{array}
\right.
\end{equation}
Analogous to (\ref{eq:2-13-0}), we can introduce a Lagrangian multiplier $q_{h}\in S_{h,s}$ and rewrite (\ref{eq:3-7}) as 
\begin{equation}\label{eq:3-8}
\left\{
\begin{array}{@{}l@{}}
{\displaystyle \omega(\omega+{\rm i}\gamma)\big( \nabla\times{\bf A}_{h},\, \nabla \times{\bf w}_{h}\big)_{s} + (\nabla q_{h},{\bf w}_{h})_{s}=  {\rm i}\omega \omega^{2}_{p}\varepsilon_{0}\big({\bf E}_{0,h} |_{\Omega_{s}},\,\nabla \times{\bf w}_{h}\big)_{s}, }\\[2mm]
{\displaystyle\qquad  \qquad \qquad \qquad \qquad \qquad \qquad \qquad \qquad \qquad \qquad \qquad {\rm for}\,\,{\rm all}\,\,{\bf w}_{h}\in Q_{h,s},}\\[0.5mm]
{\displaystyle \big({\bf A}_{h},\, \nabla \eta_{h})_{s} = 0, \quad {\rm for}\,\,{\rm all}\,\, \eta_{h} \in S_{h,s}.}
\end{array}
\right.
\end{equation}
Next taking $({\bf u}_{h}, {\bf v}_{h}) \in  X_{0,h} \times Y_{0,h,s}$ in (\ref{eq:3-1}) and using (\ref{eq:3-5}), we have 
\begin{equation}\label{eq:3-9}
\left\{
\begin{array}{@{}l@{}}
{\displaystyle  \big(\mu^{-1}\nabla\times{\bf E}_{0,h}, \,\nabla\times{\bf u}_{h}) - \omega^{2}(\varepsilon{\bf E}_{0,h}, \,{\bf u}_{h}) - {\rm i}\omega \langle {\bf E}_{0,h,T}, {\bf u}_{h,T}\rangle = \langle {\bf g}, {\bf u}_{h,T}\rangle  }\\[2mm]
{\displaystyle \qquad \qquad \qquad \qquad \qquad \qquad \qquad \qquad \quad \quad +\, {\rm i}\omega\big(\overline{{\bf J}_{0,h} + \nabla\times{\bf A}}_{h}, \,{\bf u}_{h}}\big), \\[2mm]
 {\displaystyle -\omega(\omega+{\rm i}\gamma)\big({\bf J}_{0,h} ,\,{\bf v}_{h}\big)_{s} + \beta^{2}\big(\nabla\cdot {\bf J}_{0,h},\,\nabla\cdot {\bf v}_{h}\big)_{s} = -{\rm i}\omega \omega^{2}_{p}\varepsilon_{0}\big(({\bf E}_{0,h} + \nabla \varphi_{h})|_{\Omega_{s}},\,{\bf v}_{h}\big)_{s},}
\end{array}
\right.
\end{equation}
where $\varphi_{h}$ and ${\bf A}_{h}$ satisfy (\ref{eq:3-6}) and (\ref{eq:3-8}), respectively. Paralleling the analysis of (\ref{eq:2-14}) in section~\ref{sec-2}, we recall the sesquilinear form $a_{+}$ given by (\ref{eq:2-15}) and define the discrete operator $K_{h}$: ${\bf L}^{2}(\Omega)\times{\bf L}^{2}(\Omega_{s})\rightarrow {\bf L}^{2}(\Omega)\times{\bf L}^{2}(\Omega_{s})$ such that if $\hat{{\bf u}} = ({\bf u}_{1},{\bf u}_{2}) \in {\bf L}^{2}(\Omega)\times{\bf L}^{2}(\Omega_{s})$ then $K_{h}\hat{{\bf u}} \in X_{0,h} \times Y_{0,h,s}\subset{\bf L}^{2}(\Omega)\times{\bf L}^{2}(\Omega_{s})$ satisfies
\begin{equation}\label{eq:3-10}
\begin{array}{lll}
{\displaystyle  a_{+}(K_{h}\hat{{\bf u}}, \,\hat{{\bf v}}_{h}) = -2\omega^{2}(\varepsilon{\bf u}_{1}, \,{\bf v}_{1,h}) -2\omega^{2}\big({\bf u}_{2} ,\,{\bf v}_{2,h}\big)_{s}-{\rm i}\omega \big(\overline{\bf u}_{2}, \,{\bf v}_{1,h}\big) }\\[2mm]
{\displaystyle \;+\,{\rm i}\omega \omega^{2}_{p}\varepsilon_{0}\big({\bf u}_{1}|_{\Omega_{s}},\,{\bf v}_{2,h}\big)_{s}- {\rm i}\omega\big(\overline{\nabla\times{\bf A}}_{h}, \,{\bf v}_{1,h}\big) + {\rm i}\omega \omega^{2}_{p}\varepsilon_{0}\big((\nabla \varphi_{h})|_{\Omega_{s}},\,{\bf v}_{2,h}\big)_{s}}
\end{array}
\end{equation}
${\rm for}\,\, {\rm all}\,\,\hat{{\bf v}}_{h} = ({\bf v}_{1,h},{\bf v}_{2,h}) \in X_{0,h}\times Y_{0,h,s}$, where $\varphi_{h} \in S_{h}$ and ${\bf A}_{h}\in Q_{h,s}$ are the solutions of 
\begin{equation}\label{eq:3-11}
-\omega^{2}\big(\varepsilon \nabla \varphi_{h},\,\nabla\xi_{h}\big) = {\rm i}\omega\big(\overline{\bf u}_{2}, \,\nabla\xi_{h}), \quad {\rm for}\,\,{\rm all}\,\,\xi_{h} \in S_{h},
\end{equation}
and 
\begin{equation}\label{eq:3-12}
\left\{
\begin{array}{@{}l@{}}
{\displaystyle \omega(\omega+{\rm i}\gamma)\big( \nabla\times{\bf A}_{h},\, \nabla \times{\bf w}_{h}\big)_{s} + \big(\nabla q_{h}, \, {\bf w}_{h}\big)_{s}=  {\rm i}\omega \omega^{2}_{p}\varepsilon_{0}\big({\bf u}_{1} |_{\Omega_{s}},\,\nabla \times{\bf w}_{h}\big)_{s}, }\\[2mm]
{\displaystyle\qquad  \qquad \qquad \qquad \qquad \qquad \qquad \qquad \qquad \qquad \qquad \qquad {\rm for}\,\,{\rm all}\,\,{\bf w}_{h}\in Q_{h,s},}\\[2mm]
{\displaystyle \big({\bf A}_{h},\, \nabla \eta_{h})_{s} = 0, \quad {\rm for}\,\,{\rm all}\,\, \eta_{h} \in S_{h,s}.}
\end{array}
\right.
\end{equation}
Similarly, we define a function $\hat{{\bf F}}_{h} = ({\bf F}_{1,h}, {\bf F}_{2,h}) \in  X_{0,h}\times Y_{0,h,s}$
which satisfies
\begin{equation}\label{eq:3-13}
a_{+}(\hat{{\bf F}}_{h}, \,\hat{{\bf v}}_{h}) =  \langle {\bf g}, {\bf v}_{1,h,T}\rangle, \quad {\rm for}\,\, {\rm all}\,\,\hat{{\bf v}}_{h} = ({\bf v}_{1,h},{\bf v}_{2,h}) \in X_{0,h}\times Y_{0,h,s}.
\end{equation}
Using the Lax--Milgram theorem as in section~\ref{sec-2}, we find that $K_{h}$ and ${\bf F}_{h}$ are well defined. Now we can write the system (\ref{eq:3-9}) in the form of (\ref{eq:2-27}). Find $\hat{{\bf Z}}_{h} = ({\bf E}_{0,h},\,{\bf J}_{0,h})\in X_{0,h} \times Y_{0,h,s}$ such that 
\begin{equation}\label{eq:3-14}
(I+K_{h})\hat{{\bf Z}}_{h} = \hat{{\bf F}}_{h}.
\end{equation}
In the rest of this section, we use the theory of collectively compact operators to prove the convergence of $\hat{{\bf Z}}_{h}$ to $\hat{{\bf Z}}$ as $h\rightarrow 0$, where $\hat{{\bf Z}}$ is the solution of (\ref{eq:2-27}). To this end, we need to verify the pointwise convergence of $K_{h}\hat{{\bf u}}$ to $K\hat{{\bf u}}$ in ${\bf L}^{2}(\Omega)\times{\bf L}^{2}(\Omega_{s})$ and the collective compactness of $\{K_{h}\}$ (see Chapter 2 and 7 of \cite{Monk}).
\subsection{Pointwise convergence}
In this part, we verify the pointwise convergence of $K_{h}\hat{{\bf u}}$ to $K\hat{{\bf u}}$ in ${\bf L}^{2}(\Omega)\times{\bf L}^{2}(\Omega_{s})$. We first give a useful lemma concerning the density of the finite element spaces $X_{h}$, $S_{h}$, $Y_{h,s}$, and $Q_{h,s}$.
\begin{lemma}\label{lem3-3}
The space $X_{h}$ is dense in $X(\Omega)$ in the sense that for any ${\bf u}$ in $X(\Omega)$,
\begin{equation}
\lim_{h\to 0} \inf_{{\bf u}_{h}\in X_{h}} \Vert {\bf u}_{h} - {\bf u}\Vert_{X(\Omega)} = 0.
\end{equation}
Similarly, $S_{h}$, $Y_{h,s}$ and $Q_{h,s}$ are dense in $H_{0}^{1}(\Omega)$, $Y(\Omega_{s})$, and $\mathbf{H}_{0}(\mathbf{curl};\Omega_{s})$, respectively.
\end{lemma}

This lemma can be proved by using the properties of interpolation operators in these finite element spaces and the density of smooth functions in $X(\Omega)$ and other spaces. For more details, see Lemma~7.10 of \cite{Monk}.

\begin{theorem}\label{thm:3-1}
For any function $\hat{{\bf u}}\in {\bf L}^{2}(\Omega)\times{\bf L}^{2}(\Omega_{s})$, we have 
\begin{equation*}
\Vert (K-K_{h})\hat{{\bf u}} \Vert_{X(\Omega)\times Y(\Omega_{s})} \rightarrow 0,\quad {\rm as}\;\;h\rightarrow 0. 
\end{equation*}
\end{theorem}
\begin{proof}
First we rewrite the variational problem (\ref{eq:2-17}) for $K$ as the mixed formulation. Given $\hat{{\bf u}} \in {\bf L}^{2}(\Omega)\times{\bf L}^{2}(\Omega_{s})$, find $K\hat{{\bf u}} \in X(\Omega)\times Y(\Omega_{s})$ and $\hat{{\bm \phi}} = ({\bm \phi}_{1}, {\bm \phi}_{2})\in \nabla H_{0}^{1}(\Omega) \times (\nabla \times \mathbf{H}_{0}(\mathbf{curl};\Omega_{s}))$ such that 
\begin{equation}\label{eq:3-15}
\left\{
\begin{array}{lll}
{\displaystyle  a_{+}(K\hat{{\bf u}}, \,\hat{{\bf v}}) + (\varepsilon {\bf v}_{1}, {\bm \phi}_{1}) + ({\bf v}_{2}, {\bm \phi}_{2}) = -2\omega^{2}(\varepsilon{\bf u}_{1}, \,{\bf v}_{1}) -2\omega^{2}\big({\bf u}_{2} ,\,{\bf v}_{2}\big)_{s} -{\rm i}\omega \big(\overline{\bf u}_{2}, \,{\bf v}_{1}\big) }\\[2mm]
{\displaystyle\quad +\,{\rm i}\omega \omega^{2}_{p}\varepsilon_{0}\big({\bf u}_{1}|_{\Omega_{s}},\,{\bf v}_{2}\big)_{s}-{\rm i}\omega\big(\overline{\nabla\times{\bf A}}, \,{\bf v}_{1}\big) + {\rm i}\omega \omega^{2}_{p}\varepsilon_{0}\big((\nabla \varphi)|_{\Omega_{s}},\,{\bf v}_{2}\big)_{s},}\\[2mm]
{\displaystyle \quad \quad  \qquad \qquad \qquad \quad \quad  \qquad \qquad \qquad \qquad {\rm for}\,\, {\rm all}\,\,\hat{{\bf v}} = ({\bf v}_{1},{\bf v}_{2}) \in X(\Omega) \times Y(\Omega_{s}), }\\[2mm]
{\displaystyle (\varepsilon K{\bf u}_{1}, \,{\bm \xi}_{1}) + (K{\bf u}_{2}, \,{\bm \xi}_{2}) = 0, \quad  {\rm for}\,\, {\rm all}\,\,({\bm \xi}_{1},{\bm \xi}_{2}) \in \nabla H_{0}^{1}(\Omega) \times (\nabla \times \mathbf{H}_{0}(\mathbf{curl};\Omega_{s})), }
\end{array}
\right.
\end{equation}
where $\varphi$ and ${\bf A}$ are the solutions of (\ref{eq:2-18}) and (\ref{eq:2-19}), respectively. Since $\nabla H_{0}^{1}(\Omega)$ and $\nabla \times \mathbf{H}_{0}(\mathbf{curl};\Omega_{s})$ are closed subspaces of $X(\Omega)$ and $Y(\Omega_{s})$, respectively, by taking $({\bf v}_{1},{\bf v}_{2}) = ({\bm \phi}_{1}, {\bm \phi}_{2})$ and using the fact that $\nabla \times {\bm \phi}_{1} = {\bf 0}$ and $\nabla\cdot {\bm \phi}_{2} = 0$, we have 
\begin{equation}\label{eq:3-16}
|  (\varepsilon {\bf v}_{1}, {\bm \phi}_{1}) + ({\bf v}_{2}, {\bm \phi}_{2}) | = |  (\varepsilon {\bm \phi}_{1}, {\bm \phi}_{1}) + ({\bm \phi}_{2}, {\bm \phi}_{2})| \geq C \Vert \hat{{\bm \phi}} \Vert^{2}_{X(\Omega)\times Y(\Omega_{s})},
\end{equation}
which implies the inf-sup condition.

Similarly, the discrete equation (\ref{eq:3-10}) for $K_{h}$ can be reformulated as the mixed finite element problem. Find $K_{h}\hat{{\bf u}} \in X_{h}\times Y_{h,s}$ and $\hat{{\bm \phi}}_{h} = ({\bm \phi}_{1,h}, {\bm \phi}_{2,h})\in \nabla S_{h} \times (\nabla \times Q_{h,s})$ such that 
\begin{equation}\label{eq:3-17}
\left\{
\begin{array}{lll}
{\displaystyle  a_{+}(K_{h}\hat{{\bf u}}, \,\hat{{\bf v}}_{h}) + (\varepsilon {\bf v}_{1,h}, {\bm \phi}_{1,h}) + ({\bf v}_{2,h}, {\bm \phi}_{2,h}) = -2\omega^{2}(\varepsilon{\bf u}_{1}, \,{\bf v}_{1,h}) -2\omega^{2}\big({\bf u}_{2} ,\,{\bf v}_{2,h}\big)_{s} }\\[2mm]
{\displaystyle \quad -\,{\rm i}\omega \big(\overline{\bf u}_{2}, \,{\bf v}_{1,h}\big)+{\rm i}\omega \omega^{2}_{p}\varepsilon_{0}\big({\bf u}_{1}|_{\Omega_{s}},\,{\bf v}_{2,h}\big)_{s}-{\rm i}\omega\big(\overline{\nabla\times{\bf A}}_{h}, \,{\bf v}_{1,h}\big) }\\[2mm]
{\displaystyle \quad + \,{\rm i}\omega \omega^{2}_{p}\varepsilon_{0}\big((\nabla \varphi_{h})|_{\Omega_{s}},\,{\bf v}_{2,h}\big)_{s}, \qquad \;\;\quad  {\rm for}\,\, {\rm all}\,\,\hat{{\bf v}}_{h} = ({\bf v}_{1,h},{\bf v}_{2,h}) \in X_{h} \times Y_{h,s}, }\\[2mm]
{\displaystyle (\varepsilon K_{h}{\bf u}_{1}, \,{\bm \xi}_{1,h}) + (K_{h}{\bf u}_{2}, \,{\bm \xi}_{2,h}) = 0, \qquad  {\rm for}\,\, {\rm all}\,\,({\bm \xi}_{1,h},{\bm \xi}_{2,h}) \in \nabla S_{h} \times (\nabla \times Q_{h,s}), }
\end{array}
\right.
\end{equation}
where $\varphi_{h}$ and ${\bf A}_{h}$ are the solutions of (\ref{eq:3-11}) and (\ref{eq:3-12}), respectively. Since $\nabla S_{h} \subset X_{h}$ and $\nabla \times Q_{h,s}\subset Y_{h,s}$, we have the discrete inf-sup condition analogous to (\ref{eq:3-16}).

Next we introduce an auxiliary problem of (\ref{eq:3-17}) by replacing ${\varphi}_{h}$ and ${\bf A}_{h}$ with ${\varphi}$ and ${\bf A}$ respectively. Find $K^{\prime}_{h}\hat{{\bf u}} \in X_{h}\times Y_{h,s}$ and $\hat{{\bm \phi}}^{\prime}_{h} = ({\bm \phi}^{\prime}_{1,h}, {\bm \phi}^{\prime}_{2,h})\in \nabla S_{h} \times (\nabla \times Q_{h,s})$ such that 
\begin{equation}\label{eq:3-18}
\left\{
\begin{array}{lll}
{\displaystyle  a_{+}(K^{\prime}_{h}\hat{{\bf u}}, \,\hat{{\bf v}}_{h}) + (\varepsilon {\bf v}_{1,h}, {\bm \phi}^{\prime}_{1,h}) + ({\bf v}_{2,h}, {\bm \phi}^{\prime}_{2,h}) = -2\omega^{2}(\varepsilon{\bf u}_{1}, \,{\bf v}_{1,h}) -2\omega^{2}\big({\bf u}_{2} ,\,{\bf v}_{2,h}\big)_{s} }\\[2mm]
{\displaystyle\quad -\,{\rm i}\omega \big(\overline{\bf u}_{2}, \,{\bf v}_{1,h}\big)+{\rm i}\omega \omega^{2}_{p}\varepsilon_{0}\big({\bf u}_{1}|_{\Omega_{s}},\,{\bf v}_{2,h}\big)_{s}-{\rm i}\omega\big(\overline{\nabla\times{\bf A}}, \,{\bf v}_{1,h}\big) }\\[2mm]
{\displaystyle \quad + \,{\rm i}\omega \omega^{2}_{p}\varepsilon_{0}\big((\nabla \varphi)|_{\Omega_{s}},\,{\bf v}_{2,h}\big)_{s}, \quad  \qquad\; {\rm for}\,\, {\rm all}\,\,\hat{{\bf v}}_{h} = ({\bf v}_{1,h},{\bf v}_{2,h}) \in X_{h} \times Y_{h,s}, }\\[2mm]
{\displaystyle (\varepsilon K^{\prime}_{h}{\bf u}_{1}, \, {\bm \xi}_{1,h}) + (K^{\prime}_{h}{\bf u}_{2}, \, {\bm \xi}_{2,h}) = 0, \quad  {\rm for}\,\, {\rm all}\,\,({\bm \xi}_{1,h},{\bm \xi}_{2,h}) \in \nabla S_{h} \times (\nabla \times Q_{h,s}). }
\end{array}
\right.
\end{equation}
Here $\varphi$ and ${\bf A}$ are the solutions of (\ref{eq:2-18}) and (\ref{eq:2-19}), respectively. By the theory of mixed finite element methods \cite{Boffi,Monk}, we have
\begin{equation}\label{eq:3-19}
\begin{array}{lll}
{\displaystyle \Vert (K-K^{\prime}_{h})\hat{{\bf u}} \Vert_{X(\Omega)\times Y(\Omega_{s})} + \Vert \hat{{\bm \phi}}^{\prime}_{h} - \hat{{\bm \phi}}\Vert_{X(\Omega)\times Y(\Omega_{s})} }\\[2mm]
{\displaystyle \leq C \Big\{ \inf_{\hat{{\bf u}}_{h} \in X_{h}\times Y_{h,s}}  \Vert K\hat{{\bf u}} - \hat{{\bf u}}_{h} \Vert_{X(\Omega)\times Y(\Omega_{s})}  + \inf_{\hat{{\bm \eta}}_{h} \in \nabla S_{h}\times (\nabla \times Q_{h,s})}  \Vert  \hat{{\bm \phi}} - \hat{{\bm \eta}}_{h} \Vert_{X(\Omega)\times Y(\Omega_{s})}  \Big\}.} 
\end{array}
\end{equation}
In fact, since $\hat{{\bm \phi}} \in  \nabla H_{0}^{1}(\Omega) \times (\nabla \times \mathbf{H}_{0}(\mathbf{curl};\Omega_{s}))$, and $\hat{{\bm \phi}}^{\prime}_{h}, \hat{{\bm \eta}}_{h} \in  \nabla S_{h}\times (\nabla \times Q_{h,s}) $, (\ref{eq:3-19}) is equivalent to 
\begin{equation}\label{eq:3-20}
\begin{array}{lll}
{\displaystyle \Vert (K-K^{\prime}_{h})\hat{{\bf u}} \Vert_{X(\Omega)\times Y(\Omega_{s})} + \Vert \hat{{\bm \phi}}^{\prime}_{h} - \hat{{\bm \phi}}\Vert_{{\bf L}^{2}(\Omega)\times {\bf L}^{2}(\Omega_{s})} }\\[2mm]
{\displaystyle \leq C \Big\{ \inf_{\hat{{\bf u}}_{h} \in X_{h}\times Y_{h,s}}  \Vert K\hat{{\bf u}} - \hat{{\bf u}}_{h} \Vert_{X(\Omega)\times Y(\Omega_{s})}  + \inf_{\hat{{\bm \eta}}_{h} \in \nabla S_{h}\times (\nabla \times Q_{h,s})}  \Vert  \hat{{\bm \phi}} - \hat{{\bm \eta}}_{h} \Vert_{{\bf L}^{2}(\Omega)\times {\bf L}^{2}(\Omega_{s})}  \Big\}.} 
\end{array}
\end{equation}
Now we turn to the estimates of $ (K_{h}-K^{\prime}_{h})\hat{{\bf u}}$. To this end, we subtract (\ref{eq:3-18}) from (\ref{eq:3-17}) and take $\hat{{\bf v}}_{h} =(K_{h}-K^{\prime}_{h})\hat{{\bf u}} $ in the equation. Note that $(K_{h}-K^{\prime}_{h})\hat{{\bf u}} \in X_{0,h}\times Y_{0,h,s}$. We obtain
\begin{equation}\label{eq:3-21}
\begin{array}{lll}
{\displaystyle a_{+}\big((K_{h}-K^{\prime}_{h})\hat{{\bf u}}, \, (K_{h}-K^{\prime}_{h})\hat{{\bf u}} \big) \leq C\Big \{  \Vert \nabla \times ({\bf A} - {\bf A}_{h}) \Vert_{{\bf L}^{2}(\Omega_{s})} \Vert (K_{h}-K^{\prime}_{h}){\bf u}_{1} \Vert_{{\bf L}^{2}(\Omega_{s})} }\\[2mm]
{\displaystyle \qquad  \quad +\,  \Vert \nabla ({\varphi} - {\varphi}_{h}) \Vert_{{\bf L}^{2}(\Omega)} \Vert (K_{h}-K^{\prime}_{h}) {\bf u}_{2} \Vert_{{\bf L}^{2}(\Omega_{s})} \Big\} }.
\end{array} 
\end{equation}
Since $\varphi$ and $\varphi_{h}$ satisfy (\ref{eq:2-18}) and (\ref{eq:3-11}) respectively, it is not difficult to see that
\begin{equation}\label{eq:3-22}
\Vert {\varphi} - {\varphi}_{h}  \Vert_{H^{1}(\Omega)} \leq C \inf_{\xi_{h}\in S_{h}} \Vert {\varphi} - {\xi}_{h} \Vert_{H^{1}(\Omega)}.
\end{equation}
By verifying the (discrete) inf-sup condition and the (discrete) coerciveness of (\ref{eq:2-19}) and (\ref{eq:3-12}) (see Chapter~11 of \cite{Boffi}), we can apply the theory of mixed finite element methods to obtain
\begin{equation}\label{eq:3-23-0}
\begin{array}{lll}
{\displaystyle \Vert {\bf A} - {\bf A}_{h} \Vert_{\mathbf{H}(\mathbf{curl};\Omega_{s})} + \Vert q - q_{h} \Vert_{H^{1}(\Omega_{s})} \leq C \Big \{\inf_{{\bf w}_{h} \in Q_{h,s}} \Vert {\bf A} - {\bf w}_{h} \Vert_{\mathbf{H}(\mathbf{curl};\Omega_{s})} }\\[2mm]
{\displaystyle \qquad \qquad \qquad +\,\inf_{v_{h}\in S_{h,s}} \Vert q - v_{h} \Vert_{H^{1}(\Omega_{s})} \Big\}.}
\end{array}
\end{equation}
Taking ${\bf w} = \nabla q$ in (\ref{eq:2-19}), we see that $(\nabla q, \nabla q)_{s}=0$ and thus $q=0$. It follows that 
\begin{equation}\label{eq:3-23}
\begin{array}{lll}
{\displaystyle \Vert {\bf A} - {\bf A}_{h} \Vert_{\mathbf{H}(\mathbf{curl};\Omega_{s})} \leq C \inf_{{\bf w}_{h} \in Q_{h,s}} \Vert {\bf A} - {\bf w}_{h} \Vert_{\mathbf{H}(\mathbf{curl};\Omega_{s})} .}
\end{array}
\end{equation}
Substituting (\ref{eq:3-22}) and (\ref{eq:3-23}) into (\ref{eq:3-21}) and using the coerciveness of $a_{+}$, we come to
\begin{equation}\label{eq:3-24}
\Vert (K_{h}-K^{\prime}_{h})\hat{{\bf u}} \Vert _{{X(\Omega)\times Y(\Omega_{s})} } \leq C \Big \{ \inf_{\xi_{h}\in S_{h}} \Vert {\varphi} - {\xi}_{h} \Vert_{H^{1}(\Omega)} +   \inf_{{\bf w}_{h} \in Q_{h,s}} \Vert {\bf A} - {\bf w}_{h} \Vert_{\mathbf{H}(\mathbf{curl};\Omega_{s})} \Big\}.
\end{equation}
Combing (\ref{eq:3-20}) and (\ref{eq:3-24}), we end up with
\begin{equation}\label{eq:3-25}
\begin{array}{lll}
{\displaystyle \Vert (K-K_{h})\hat{{\bf u}} \Vert_{X(\Omega)\times Y(\Omega_{s})} \leq C \Big \{ \inf_{\xi_{h}\in S_{h}} \Vert {\varphi} - {\xi}_{h} \Vert_{H^{1}(\Omega)} +   \inf_{{\bf w}_{h} \in Q_{h,s}} \Vert {\bf A} - {\bf w}_{h} \Vert_{\mathbf{H}(\mathbf{curl};\Omega_{s})} }\\[2mm]
{\displaystyle  + \,\inf_{\hat{{\bf u}}_{h} \in X_{h}\times Y_{h,s}}  \Vert K\hat{{\bf u}} - \hat{{\bf u}}_{h} \Vert_{X(\Omega)\times Y(\Omega_{s})}  + \inf_{\hat{{\bm \eta}}_{h} \in \nabla S_{h}\times (\nabla \times Q_{h,s})}  \Vert  \hat{{\bm \phi}} - \hat{{\bm \eta}}_{h} \Vert_{{\bf L}^{2}(\Omega)\times {\bf L}^{2}(\Omega_{s})}  \Big\}.} 
\end{array}
\end{equation}
Since $X_{h},\, S_{h}, \,Y_{h,s}$ and $Q_{h,s}$ are dense in $X(\Omega), \,H_{0}^{1}(\Omega),\,Y(\Omega_{s})$ and $\mathbf{H}_{0}(\mathbf{curl};\Omega_{s})$ respectively, we have $\Vert (K-K_{h})\hat{{\bf u}} \Vert_{X(\Omega)\times Y(\Omega_{s})}\rightarrow 0 $ as $h\rightarrow 0 $. \qquad \end{proof}

Next we estimate $\hat{{\bf F}} - \hat{{\bf F}}_{h}$ by a similar argument. 
\begin{lemma}
Let $\hat{{\bf F}}$ and $\hat{{\bf F}}_{h}$ be defined by (\ref{eq:2-25}) and (\ref{eq:3-13}), respectively. We have
\begin{equation}\label{eq:3-26}
\Vert \hat{{\bf F}} -  \hat{{\bf F}}_{h}\Vert_{X(\Omega)\times Y(\Omega_{s})} \leq C
\inf_{\hat{{\bf u}}_{h} \in X_{h}\times Y_{h,s}}  \Vert \hat{{\bf F}} - \hat{{\bf u}}_{h} \Vert_{X(\Omega)\times Y(\Omega_{s})} .
\end{equation}
\end{lemma}
\begin{proof}
We rewrite (\ref{eq:2-25}) and (\ref{eq:3-13}) as the mixed problems. Finding $\hat{{\bf F}}\in X(\Omega)\times Y(\Omega_{s})$ and $\hat{{\bm \phi}} = ({\bm \phi}_{1}, {\bm \phi}_{2})\in \nabla H_{0}^{1}(\Omega) \times (\nabla \times \mathbf{H}_{0}(\mathbf{curl};\Omega_{s}))$ such that 
\begin{equation}\label{eq:3-27}
\left\{
\begin{array}{lll}
{\displaystyle  a_{+}(\hat{{\bf F}}, \,\hat{{\bf v}}) + (\varepsilon {\bf v}_{1}, {\bm \phi}_{1}) + ({\bf v}_{2}, {\bm \phi}_{2}) =  \langle {\bf g}, {\bf v}_{1,T}\rangle, \quad  {\rm for}\,\, {\rm all}\,\,\hat{{\bf v}} = ({\bf v}_{1},{\bf v}_{2}) \in X(\Omega) \times Y(\Omega_{s}), }\\[2mm]
{\displaystyle (\varepsilon {\bf F}_{1}, {\bm \xi}_{1}) + ({\bf F}_{2}, {\bm \xi}_{2}) = 0, \quad  \quad {\rm for}\,\, {\rm all}\,\,({\bm \xi}_{1},{\bm \xi}_{2}) \in \nabla H_{0}^{1}(\Omega) \times (\nabla \times \mathbf{H}_{0}(\mathbf{curl};\Omega_{s})), }
\end{array}
\right.
\end{equation}
and finding $\hat{{\bf F}}_{h}\in X_{h}\times Y_{h,s}$ and $\hat{{\bm \phi}}_{h} = ({\bm \phi}_{1,h}, {\bm \phi}_{2,h})\in \nabla S_{h} \times (\nabla \times Q_{h,s})$ such that 
\begin{equation}
\left\{
\begin{array}{lll}
{\displaystyle  a_{+}(\hat{{\bf F}}_{h}, \,\hat{{\bf v}}_{h}) + (\varepsilon {\bf v}_{1,h}, {\bm \phi}_{1,h}) + ({\bf v}_{2,h}, {\bm \phi}_{2,h}) =  \langle {\bf g}, {\bf v}_{1,h,T}\rangle, }\\[2mm]
{\displaystyle  \qquad \qquad\qquad \qquad \qquad  \;\;\qquad \qquad{\rm for}\,\, {\rm all}\,\,\hat{{\bf v}}_{h} = ({\bf v}_{1,h},{\bf v}_{2,h}) \in X_{h}\times Y_{h,s}, }\\[2mm]
{\displaystyle (\varepsilon {\bf F}_{1,h}, {\bm \xi}_{1,h}) + ({\bf F}_{2,h}, {\bm \xi}_{2,h}) = 0, \quad  {\rm for}\,\, {\rm all}\,\,({\bm \xi}_{1,h},{\bm \xi}_{2,h}) \in \nabla S_{h}\times (\nabla \times Q_{h,s}). }
\end{array}
\right.
\end{equation}
Reasoning as before, we get
\begin{equation}\label{eq:3-28}
\begin{array}{lll}
{\displaystyle \Vert \hat{{\bf F}} -  \hat{{\bf F}}_{h}\Vert_{X(\Omega)\times Y(\Omega_{s})} +  \Vert \hat{{\bm \phi}}_{h} - \hat{{\bm \phi}}\Vert_{{\bf L}^{2}(\Omega)\times {\bf L}^{2}(\Omega_{s})} }\\[2mm]
{\displaystyle  \leq C \Big\{\inf_{\hat{{\bf u}}_{h} \in X_{h}\times Y_{h,s}}  \Vert \hat{{\bf F}} - \hat{{\bf u}}_{h} \Vert_{X(\Omega)\times Y(\Omega_{s})} + \inf_{\hat{{\bm \eta}}_{h} \in \nabla S_{h}\times (\nabla \times Q_{h,s})}  \Vert  \hat{{\bm \phi}} - \hat{{\bm \eta}}_{h} \Vert_{{\bf L}^{2}(\Omega)\times {\bf L}^{2}(\Omega_{s})} \Big\}.}
\end{array}
\end{equation}
Taking $\hat{{\bf v}} = \hat{{\bm \phi}}$ in (\ref{eq:3-27}) and using the fact that $\hat{{\bf F}} \in X_{0}(\Omega)\times Y_{0}(\Omega_{s})$, we see that $(\varepsilon {\bm \phi}_{1}, {\bm \phi}_{1}) + ({\bm \phi}_{2}, {\bm \phi}_{2}) =0 $ and thus ${\bm \phi}_{1} = {\bm \phi}_{2} = 0$. Hence (\ref{eq:3-26}) follows from (\ref{eq:3-28}).\quad \end{proof}

\subsection{Collective compactness}In this part, we verify the collective compactness of $\{K_{h} \}$. To begin with, we give the definition of collective compactness.
\begin{definition}
Let $X$ be a Hilbert space and $\mathcal{K}=\{K_{n}:\,\mathcal{X}\rightarrow \mathcal{X}, \;\;n=0,1,2,\cdots\}$ be a set of bounded linear operators.
If for each bounded set $\mathcal{U}\subset \mathcal{X}$, the image set
\begin{equation}
\mathcal{K}(\mathcal{U})=\{ K_{n}u | \;{\rm for}\, {\rm each} \,u \in \mathcal{U},\;{\rm and} \,K_{n}\in \mathcal{K} \}
\end{equation}
is relatively compact, then the set $\mathcal{K}$ is called collectively compact.
\end{definition}

Let $\wedge = \{ h_{n} \}_{n=1}^{\infty}$ be a sequence of decreasing mesh size satisfying that $h_{n}\rightarrow 0 $ as $n\rightarrow \infty$. We need to show that $\{K_{h}\}_{h\in \wedge}$ is a collectively compact set of operators.. To this end, it suffices to show that the finite element spaces $X_{0,h}$ and $Y_{0,h,s}$ have the discrete compactness properties defined as follows.
\begin{definition}
$X_{0,h}$ $(Y_{0,h,s})$ is said to have the discrete compactness property, if for every sequence $\{ {\bf u}_{h} \}_{h\in \wedge}$ satisfying
\begin{itemize}
\item ${\bf u}_{h}\in X_{0,h} \,(Y_{0,h,s})$ for each $h\in \wedge$;
\item there is a constant $C$ independent of ${\bf u}_{h}$ such that $\Vert {\bf u}_{h}\Vert_{X(\Omega)}\leq C$ $(\Vert {\bf u}_{h}\Vert_{Y(\Omega_{s})}\leq C)$,
\end{itemize}
then there exists a subsequence, still denoted $\{{\bf u}_{h}\}$, and a function ${\bf u} \in X_{0}(\Omega)$ $(Y_{0}(\Omega_{s}))$ such that 
\begin{equation}
{\bf u}_{h}\rightarrow {\bf u} \;\, {\rm strongly}\; {\rm in} \;{\bf L}^{2}(\Omega) \,({\bf L}^{2}(\Omega_{s})) \;{\rm as}\; h \rightarrow 0 \;{\rm in}\; \wedge.
\end{equation}
\end{definition}
We have the following result.
\begin{theorem}\label{thm4-0}
If $X_{0,h}$ and $Y_{0,h,s}$ $(h\in \wedge)$ have the discrete compactness properties, then $\{K_{h}: {\bf L}^{2}(\Omega)\times {\bf L}^{2}(\Omega_{s})\rightarrow  {\bf L}^{2}(\Omega)\times {\bf L}^{2}(\Omega_{s}),\;\; h\in \wedge\}$ is a collectively compact set of operators.
\end{theorem}

The proof of this theorem is exactly the same as that of Theorem~7.14 of \cite{Monk} and we refer the reader to it. It remains to prove the discrete compactness properties of $\{X_{0,h}\}_{h\in \wedge}$ and $\{Y_{0,h,s}\}_{h\in \wedge}$. In fact, the discrete compactness property of $\{X_{0,h}\}_{h\in \wedge}$ has been proved in Chapter~4 of \cite{Monk} and we can use the similar trick to prove that for $\{Y_{0,h,s}\}_{h\in \wedge}$. Since the proof is very tedious and lengthy, due to the limitation of space, we omit it here.

Having verified the pointwise convergence and the collective compactness of $\{ K_{h}\}_{h \in \wedge}$, by the theory of collectively compact operators (Theorem~2.51 of \cite{Monk}), we have the ${\bf L}^{2}$ convergence of the solutions of the discrete equation (\ref{eq:3-14}).
\begin{theorem}
Let $\mathcal{T}_{h}$ be a quasi-uniform mesh. For $h\in \wedge$ sufficiently small, the discrete equation (\ref{eq:3-14}) exists a unique solution $\hat{{\bf Z}}_{h} = ({\bf E}_{0,h},\,{\bf J}_{0,h})\in X_{0,h} \times Y_{0,h,s}$ with the following error estimate
\begin{equation}\label{eq:3-29}
\Vert \hat{{\bf Z}}_{h} - \hat{{\bf Z}} \Vert_{{\bf L}^{2}(\Omega)\times {\bf L}^{2}(\Omega_{s})} \leq C \big \{ \Vert \hat{{\bf F}}_{h} - \hat{{\bf F}}  \Vert_{{\bf L}^{2}(\Omega)\times {\bf L}^{2}(\Omega_{s})} + \Vert (K-K_{h})\hat{{\bf Z}} \Vert_{{\bf L}^{2}(\Omega)\times {\bf L}^{2}(\Omega_{s})} \big\},
\end{equation}
where $\hat{{\bf Z}} = ({\bf E}_{0},{\bf J}_{0})$ is the solution of (\ref{eq:2-27}), $\hat{{\bf F}}$ and $\hat{{\bf F}}_{h}$ are the solutions of (\ref{eq:2-25}) and (\ref{eq:3-13}), respectively.
\end{theorem} 

Now we can prove the convergence of the finite element solution $({\bf E}_{h},{\bf J}_{h})$ to the solution $({\bf E},{\bf J})$ of the continuous problem (\ref{eq:2-4}) in $X(\Omega)\times Y(\Omega_{s})$.
\begin{theorem}
Suppose that $\mathcal{T}_{h}$ is a quasi-uniform mesh and $h\in \wedge$ is sufficiently small. The finite element approximation of the continuous system (\ref{eq:2-4}) given by (\ref{eq:3-1}) has a unique solution $({\bf E}_{h}, {\bf J}_{h}) \in X_{h}\times Y_{h,s}$ with the following error estimate
\begin{equation}\label{eq:3-30}
\begin{array}{lll}
{\displaystyle \Vert {\bf E}-{\bf E}_{h} \Vert_{X(\Omega)} + \Vert {\bf J} - {\bf J}_{h} \Vert_{Y(\Omega_{s})} \leq C \Big\{ \inf_{{\xi}_{h}\in S_{h}} \Vert \xi_{h} - \varphi \Vert_{H^{1}(\Omega)}+ \inf_{{\bf w}_{h}\in Q_{h,s}} \Vert {\bf w}_{h}- \mathbf{A}\Vert_{\mathbf{H}(\mathbf{curl};\Omega_{s})}   }\\[2mm]
{\displaystyle \quad +\,\inf_{\hat{{\bf u}}_{h} \in X_{h}\times Y_{h,s}}  \Vert \hat{{\bf F}} - \hat{{\bf u}}_{h} \Vert_{X(\Omega)\times Y(\Omega_{s})} + \inf_{\hat{{\bf v}}_{h} \in X_{h}\times Y_{h,s}}  \Vert K\hat{{\bf Z}} - \hat{{\bf v}}_{h} \Vert_{X(\Omega)\times Y(\Omega_{s})}    }\\[2mm]
{\displaystyle \quad +\, \inf_{\hat{{\bm \eta}}_{h} \in \nabla S_{h}\times (\nabla \times Q_{h,s})}  \Vert  \hat{{\bm \phi}} - \hat{{\bm \eta}}_{h} \Vert_{{\bf L}^{2}(\Omega)\times {\bf L}^{2}(\Omega_{s})} \Big\}, }
\end{array}
\end{equation}
where $\varphi$ and ${\bf A}$ are the solutions of (\ref{eq:2-18}) and (\ref{eq:2-19}), respectively, and $\big(K\hat{{\bf Z}}, \,\hat{{\bm \phi}}\big)$ is the solution of the mixed problem (\ref{eq:3-15}) with $\hat{\bf u}$ replaced by $\hat{{\bf Z}} = ({\bf E}_{0}, {\bf J}_{0}) \in X_{0}(\Omega) \times Y_{0}(\Omega_{s})$.
\end{theorem}
\begin{proof}
Recalling the Helmholtz decompositions (\ref{eq:2-10}) and (\ref{eq:3-3}) for $({\bf E}, {\bf J})$ and $({\bf E}_{h}, {\bf J}_{h})$ respectively, we have 
\begin{equation}\label{eq:3-31}
\begin{array}{lll}
{\displaystyle   \Vert {\bf E}-{\bf E}_{h} \Vert_{X(\Omega)} + \Vert {\bf J} - {\bf J}_{h} \Vert_{Y(\Omega_{s})} \leq  \Vert {\bf E}_{0}-{\bf E}_{0,h} \Vert_{X(\Omega)} + \Vert {\bf J}_{0} - {\bf J}_{0,h} \Vert_{Y(\Omega_{s})} }\\[2mm]
{\displaystyle \qquad  + \,\Vert \nabla (\varphi - \varphi_{h})\Vert_{{\bf L}^{2}(\Omega)} +  \Vert \nabla \times  ({\bf A}- {\bf A}_{h})\Vert_{{\bf L}^{2}(\Omega_{s})}}.
\end{array}
\end{equation}
Since $\varphi$ and $\varphi_{h}$ satisfy (\ref{eq:2-12-0}) and (\ref{eq:3-6}), respectively, we see that
\begin{equation}\label{eq:3-32}
\Vert \nabla (\varphi - \varphi_{h})\Vert_{{\bf L}^{2}(\Omega)} \leq \Vert {\bf J}_{0} - {\bf J}_{0,h} \Vert_{{\bf L}^{2}(\Omega_{s})} + C \inf_{{\xi}_{h}\in S_{h}} \Vert \xi_{h} - \varphi \Vert_{H^{1}(\Omega)}.
\end{equation}
Similarly, applying the mixed finite element theory to (\ref{eq:2-13-0}) and (\ref{eq:3-8}), it follows that
\begin{equation}\label{eq:3-33}
 \Vert \nabla \times  ({\bf A}- {\bf A}_{h})\Vert_{{\bf L}^{2}(\Omega_{s})} \leq \Vert {\bf E}_{0} - {\bf E}_{0,h} \Vert_{{\bf L}^{2}(\Omega)} + C \inf_{{\bf w}_{h}\in Q_{h,s}} \Vert {\bf w}_{h}- \mathbf{A}\Vert_{\mathbf{H}(\mathbf{curl};\Omega_{s})}.
\end{equation}
Substituting (\ref{eq:3-32}) and (\ref{eq:3-33}) into (\ref{eq:3-31}), we obtain
\begin{equation}\label{eq:3-34}
\begin{array}{lll}
{\displaystyle   \Vert {\bf E}-{\bf E}_{h} \Vert_{X(\Omega)} + \Vert {\bf J} - {\bf J}_{h} \Vert_{Y(\Omega_{s})} \leq  C\Big \{ \Vert {\bf E}_{0}-{\bf E}_{0,h} \Vert_{X(\Omega)} + \Vert {\bf J}_{0} - {\bf J}_{0,h} \Vert_{Y(\Omega_{s})}  }\\[2mm]
{\displaystyle \quad  + \,\inf_{{\xi}_{h}\in S_{h}} \Vert \xi_{h} - \varphi \Vert_{H^{1}(\Omega)} +  \inf_{{\bf w}_{h}\in Q_{h,s}} \Vert {\bf w}_{h}- \mathbf{A}\Vert_{\mathbf{H}(\mathbf{curl};\Omega_{s})} \Big \} }.
\end{array}
\end{equation}
Recalling the equations (\ref{eq:2-27}) and (\ref{eq:3-14}) for $\hat{{\bf Z}} = ({\bf E}_{0}, {\bf J}_{0}) $ and $\hat{{\bf Z}}_{h} = ({\bf E}_{0,h}, {\bf J}_{0,h}) $, respectively, we get
\begin{equation}\label{eq:3-35}
\begin{array}{lll}
{\displaystyle \Vert {\bf E}_{0}-{\bf E}_{0,h} \Vert_{X(\Omega)} + \Vert {\bf J}_{0} - {\bf J}_{0,h} \Vert_{Y(\Omega_{s})} = \Vert \hat{{\bf Z}} - \hat{{\bf Z}}_{h} \Vert_{X(\Omega)\times Y(\Omega_{s})} }\\[2mm]
{\displaystyle \quad \leq  \Vert (K_{h} -K) \hat{{\bf Z}} \Vert _{X(\Omega)\times Y(\Omega_{s})}+\Vert K_{h} (\hat{{\bf Z}}_{h} - \hat{{\bf Z}}) \Vert _{X(\Omega)\times Y(\Omega_{s})}  + \Vert \hat{{\bf F}} - \hat{{\bf F}}_{h} \Vert_{X(\Omega)\times Y(\Omega_{s})} }\\[2mm]
{\displaystyle \quad  \leq \Vert (K_{h} -K) \hat{{\bf Z}} \Vert _{X(\Omega)\times Y(\Omega_{s})}+C \Vert \hat{{\bf Z}}_{h} - \hat{{\bf Z}} \Vert _{{\bf L}^{2}(\Omega)\times {\bf L}^{2}(\Omega_{s})}  + \Vert \hat{{\bf F}} - \hat{{\bf F}}_{h} \Vert_{X(\Omega)\times Y(\Omega_{s})}, }
\end{array}
\end{equation}
where we have used the uniform continuity of $K_{h}$. Inserting (\ref{eq:3-29}) into (\ref{eq:3-35}), we see that
\begin{equation}\label{eq:3-36}
\begin{array}{lll}
{\displaystyle \Vert {\bf E}_{0}-{\bf E}_{0,h} \Vert_{X(\Omega)} + \Vert {\bf J}_{0} - {\bf J}_{0,h} \Vert_{Y(\Omega_{s})} \leq C\Big \{ \Vert (K_{h} -K) \hat{{\bf Z}} \Vert _{X(\Omega)\times Y(\Omega_{s})} + \Vert \hat{{\bf F}} - \hat{{\bf F}}_{h} \Vert_{X(\Omega)\times Y(\Omega_{s})} \Big \}}.
\end{array}
\end{equation}
Using (\ref{eq:3-25}) and (\ref{eq:3-26}) with $\hat{{\bf u}}$ replaced by $\hat{{\bf Z}} = ({\bf E}_{0}, {\bf J}_{0})$ in Theorem~\ref{thm:3-1}, we further have
\begin{equation}\label{eq:3-37}
\begin{array}{lll}
{\displaystyle \Vert {\bf E}_{0}-{\bf E}_{0,h} \Vert_{X(\Omega)} + \Vert {\bf J}_{0} - {\bf J}_{0,h} \Vert_{Y(\Omega_{s})} \leq C \Big\{ \inf_{{\xi}_{h}\in S_{h}} \Vert \xi_{h} - \varphi \Vert_{H^{1}(\Omega)}+ \inf_{{\bf w}_{h}\in Q_{h,s}} \Vert {\bf w}_{h}- \mathbf{A}\Vert_{\mathbf{H}(\mathbf{curl};\Omega_{s})}   }\\[2mm]
{\displaystyle \quad +\,\inf_{\hat{{\bf u}}_{h} \in X_{h}\times Y_{h,s}}  \Vert \hat{{\bf F}} - \hat{{\bf u}}_{h} \Vert_{X(\Omega)\times Y(\Omega_{s})} + \inf_{\hat{{\bf v}}_{h} \in X_{h}\times Y_{h,s}}  \Vert K\hat{{\bf Z}} - \hat{{\bf v}}_{h} \Vert_{X(\Omega)\times Y(\Omega_{s})}    }\\[2mm]
{\displaystyle \quad +\, \inf_{\hat{{\bm \eta}}_{h} \in \nabla S_{h}\times (\nabla \times Q_{h,s})}  \Vert  \hat{{\bm \phi}} - \hat{{\bm \eta}}_{h} \Vert_{{\bf L}^{2}(\Omega)\times {\bf L}^{2}(\Omega_{s})} \Big\}. }
\end{array}
\end{equation}
Substituting (\ref{eq:3-37}) into (\ref{eq:3-34}), we obtain the desired result (\ref{eq:3-30}) and complete the proof of this theorem. \qquad \end{proof}

\begin{rem}
By the density of the finite element spaces (Lemma~\ref{lem3-3}), we see that $ \Vert {\bf E}-{\bf E}_{h} \Vert_{X(\Omega)} + \Vert {\bf J} - {\bf J}_{h} \Vert_{Y(\Omega_{s})} \rightarrow 0$ as $h\rightarrow 0$. Furthermore, if the solutions $\varphi$, $\mathbf{A}$, $\hat{{\bf F}}$, $K\hat{{\bf Z}}$, and $\hat{{\bm \phi}}$ possess higher regularity, it is possible to obtain an explicit convergence rate for the finite element approximation by using error estimates of interpolation operators.
\end{rem}

\section{Numerical examples}\label{sec-5}

\subsection{Convergence study}
In this section, we perform numerical tests to validate the proposed finite element scheme and confirm our theoretical analysis. To this end, we consider an artificial problem
 \begin{equation*}
\left\{
\begin{array}{@{}l@{}}
{\displaystyle  \nabla \times (\nabla\times {\bf E}) - \omega^{2} {\bf E} - {\rm i}\omega {\bf J} = \mathbf{f}_{1}, \quad {\rm in} \;\; \Omega,}\\[2mm]
 {\displaystyle \omega(\omega+{\rm i}\gamma){\bf J} + \beta^{2} \nabla(\nabla\cdot {\bf J}) - {\rm i}\omega \omega^{2}_{p} {\bf E} = \mathbf{f}_{2}, \quad {\rm in}\;\; \Omega_{s},}\\[2mm]
 {\displaystyle (\nabla\times {\bf E})\times {\bf n} - {\rm i}\omega({\bf n}\times{\bf E})\times{\bf n} = {\bf g},\quad {\rm on}\;\; \partial \Omega,}\\[2mm]
 {\displaystyle {\bf n}\cdot {\bf J} = 0,\quad  {\rm on}\;\; \partial \Omega_{s}.}
\end{array}
\right.
\end{equation*}
We take $\Omega = \Omega_{s} = (0,1)^{3}$. In addition, we set $\omega = \omega_{p} = \gamma = \beta = 1$. The right-hand terms $\mathbf{f}_{i} \,(i=1,2)$ and the boundary condition $\mathbf{g}$ are chosen such that the problem has the following exact solution
\begin{equation*}
\mathbf{E} = \big(e^{-{\rm i}z}, \,   0,\, 0\big), \quad \mathbf{J} = \big(\sin(\pi x),\,\sin(\pi y),\, {\rm i} \sin(\pi z)\big).
\end{equation*}

We solve the problem by the proposed finite element scheme (\ref{eq:3-1}) with linear elements $(r=1)$ and quadratic elements $(r=2)$, respectively. We analyze the convergence of the method on a sequence of successively refined tetrahedral meshes starting from a coarse mesh. Numerical results of the linear element method and the quadratic element method are presented in Table~5.1 and~5.2, respectively.
We observe that the proposed finite element method with linear elements or quadratic elements has an optimal convergence order for the electric field and current density.

\begin{table}[htbp]\label{table:5-1}
\centering
\caption{Convergence results of the linear element method. $N$ is degrees of freedom.}
\begin{tabular}{ccccccc}
  \hline$h$ & $N$& $\|\mathbf{J}-\mathbf{J}_{h}\|_{\mathbf{H}(\mathbf{div})}$ & Order &  $N$ &$\|\mathbf{E}-\mathbf{E}_{h}\|_{\mathbf{H}_{T}(\mathbf{curl})}$ &Order\\
 \hline
  $h   $ & $720$ & $ 3.131$E$-1$ &$- $   &$196$ & $1.143$E$-1$&$-$\\
  $h/2 $ & $5184$ & $ 1.594$E$-1$ &$ 0.974$& $1208$ & $5.612$E$-2$&$1.027$\\
  $h/4 $ & $39168$ & $ 8.004$E$-2$ &$ 0.994$&  $8368$ &$2.776$E$-2$&$1.016$\\
  $h/8 $ & $304128$ & $ 4.007$E$-2$ &$0.998$& $62048$ & $1.381$E$-2$&$1.007$\\
  $h/16 $ & $2396160$ & $  2.004$E$-2$ &$1.000$&  $477376$ &$6.895$E$-3$&$1.002$\\
  \hline
\end{tabular}
\end{table}

\begin{table}[htbp]\label{table:5-2}
\centering
\caption{Convergence results of the quadratic element method. $N$ is degrees of freedom.}
\begin{tabular}{ccccccc}
  \hline$h$ & $N$& $\|\mathbf{J}-\mathbf{J}_{h}\|_{\mathbf{H}(\mathbf{div})}$ & Order &  $N$ &$\|\mathbf{E}-\mathbf{E}_{h}\|_{\mathbf{H}_{T}(\mathbf{curl})}$ &Order\\
 \hline
  $h   $ & $2016$ & $ 5.684$E$-2$ &$- $   &$1308$ & $6.524$E$-3$&$-$\\
  $h/2 $ & $14976$ & $ 1.446$E$-2$ &$ 1.975$& $8808$ & $1.646$E$-3$&$1.987$\\
  $h/4 $ & $115200$ & $ 3.630$E$-3$ &$ 1.994$&  $64272 $ &$4.147$E$-4$&$1.989$\\
  $h/8 $ & $903168$ & $ 9.084$E$-4$ &$1.998$& $490272$ & $1.041$E$-4$&$1.994$\\
  $h/16 $ & $7151616$ & $  2.273$E$-4$ &$1.999$&  $3828288$ &$2.608$E$-5$&$1.997$\\
  %$h/32 $ & $ 37896192$ & $ 5.083$E$-2$ &$0.997$&  $7414144$ &$1.258$E$-2$&$0.998$\\
  \hline
\end{tabular}
\end{table}

\subsection{Scattering of a single metal nanosphere}
In this section, we consider a more physical problem in nanophotonics, i.e., the scattering of a plane wave from a single metal nanosphere in free space. The radius of the nanosphere is 2\,nm and the Silver--M\"{u}ller condition is set on the boundary of a concentric sphere of radius 20\,nm. The nanosphere is irradiated by a plane wave $\mathbf{E}^{inc} = {\rm exp}({\rm i}\omega y) \mathbf{e}_{x}$ propagating in the $y$-direction. Physical parameters for the NHD model are summarized in Table~5.3. The domain is partitioned into 26510080 tetrahedrons and the electric field and current density are approximated by the first order curl- and divergence-conforming elements, respectively.

\begin{table}[htbp]\label{table:5-3}
\centering
\caption{Physical parameters for the NHD model. $\epsilon_{0}$ and $\mu_{0}$ are respectively the electric permittivity and magnetic permeability of free space.}
\begin{tabular}{ccccc}
  \hline$\omega_{p}$ & $\gamma$& $\beta$ & $\epsilon$ &  $\mu$ \\
 \hline
  $8.65\times 10^{15}$ rad/s & $8.65 \times 10^{13}$ rad/s & $8.29\times 10^{5}$ m/s &$\epsilon_{0}$&$\mu_{0}$   \\
  \hline
\end{tabular}
\end{table}

The extinction cross section ($\sigma_{\rm ext}$) measures the total losses of energy from the incident wave due to both absorption and scattering by the scatterer, which is a quantity of interest to physicists. In this example, $\sigma_{\rm ext}$ is defined as 
\begin{equation}
\sigma_{\rm ext} = -\frac{1}{D|\mathbf{E}_{0}|^{2}} \oint_{\mathcal{S}} {\rm Re}\,[\mathbf{E}^{inc}\times \mathbf{H}_{s}^{\ast}+\mathbf{E}_{s}\times (\mathbf{H}^{inc})^{\ast}]\cdot{\bf n} \, dS,
\end{equation}
where $D$ is the diameter of the nanosphere, $\mathcal{S}$ is the surface of the nanosphere, $\mathbf{E}_{s}$ and $\mathbf{H}_{s}$ are the scattered fields satisfying 
\begin{equation*}
\mathbf{E}_{s} = {\bf E}- \mathbf{E}^{inc}, \quad \mathbf{H}_{s} = {\bf H}- \mathbf{H}^{inc}.
\end{equation*}
In Fig~5.1 we plot the extinction cross section $\sigma_{ext}$ at different angular frequencies $\omega$.

In Fig~\ref{Fig:5-2} we display the electric-field and current-density distributions in a section of the nanosphere at the resonant frequency $\omega/\omega_{p} = 1.13$.

\begin{figure}\label{Fig:5-1}
\begin{center}
\includegraphics[width=7.2cm,height=6cm]{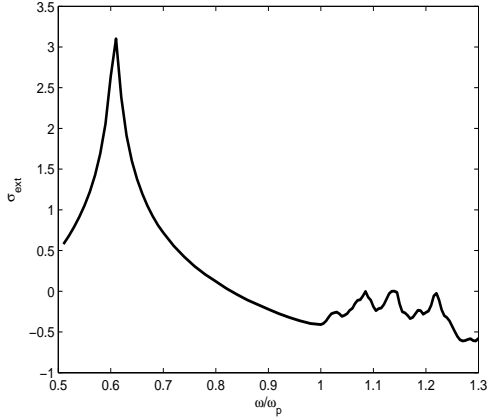}
\caption{Extinction cross section $\sigma_{ext}$ of a nanosphere with a radius of 2\,nm irradiated by a plane wave.}
\end{center}
\end{figure}

\begin{figure}\label{Fig:5-2}
\begin{minipage}{0.48\linewidth}
\centerline{\includegraphics[width=6.2cm, height = 5cm]{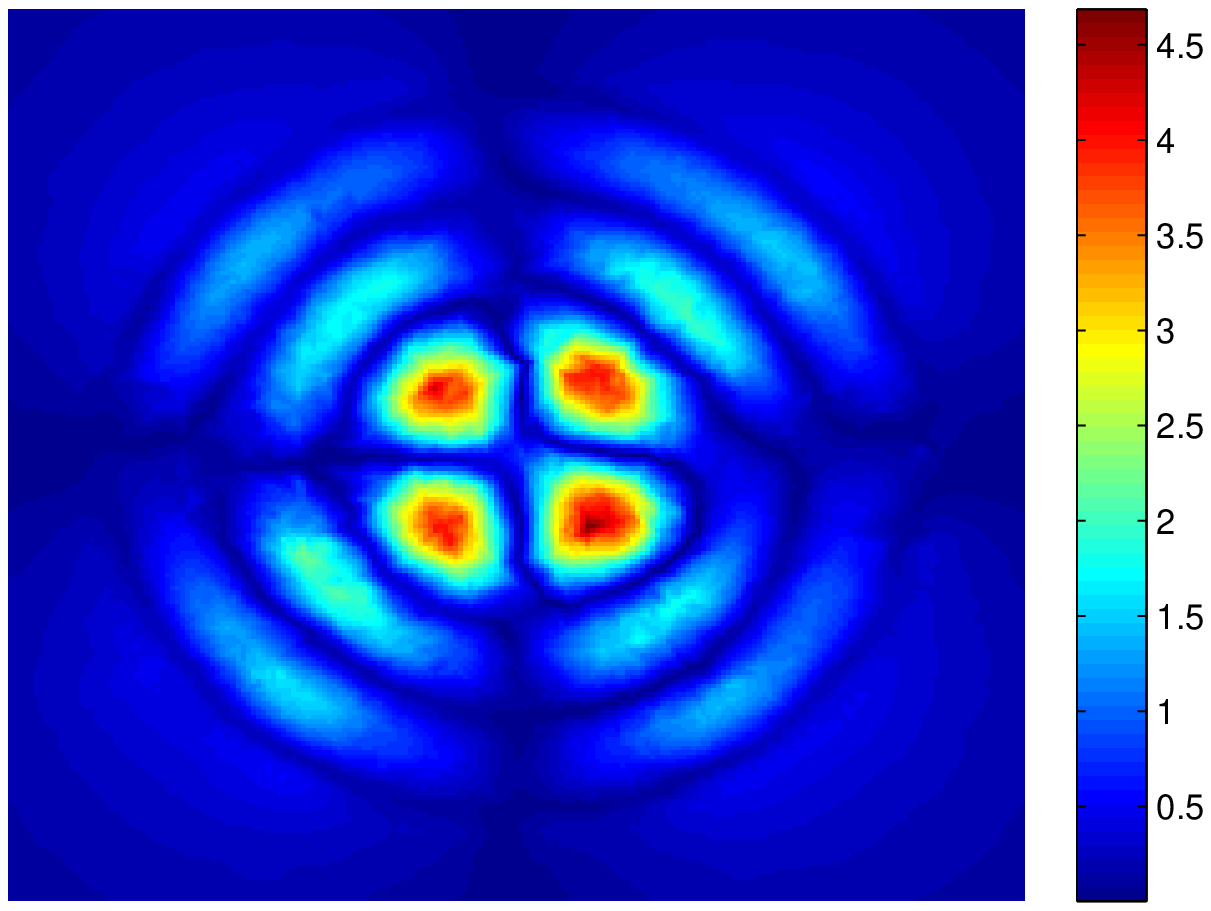}}
\centerline{(a):$\;|\mathbf{E}_{z}|$}
\end{minipage}
\qquad
\begin{minipage}{0.48\linewidth}
\centerline{\includegraphics[width=6.2cm, height = 5cm]{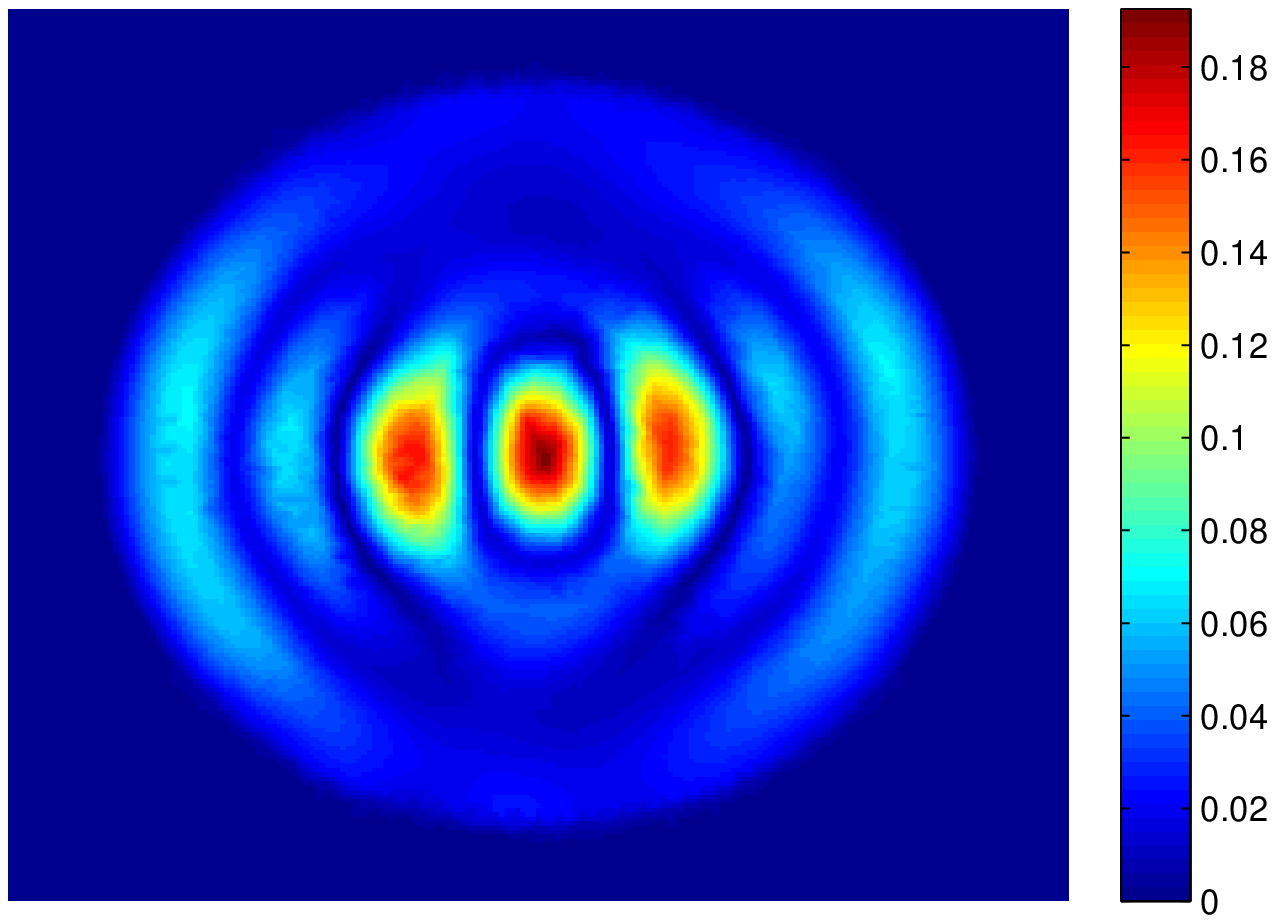}}
\centerline{(b):$\;|\mathbf{J}_{x}|$}
\end{minipage}
\caption{The electric-field and current-density distributions at $\omega/\omega_{p} = 1.13$. }
\end{figure}

\section{Conclusions}
We have given the first mathematical and numerical analysis of the frequency-domain NHD model. The existence and uniqueness of solutions to the weak formulation of the equations are proved. The convergence of a Galerkin finite element scheme is proved by using the theory of collectively compact operators. Numerical tests are presented to validate the finite element scheme and confirm the theoretical analysis. 

Building on the work in this paper, in the near future we plan to investigate the preconditioning of the linear system resulting from the finite element approximation. Since this linear system is indefinite, an efficient preconditioner is essential for solving it by an iterative method. Moreover, developing efficient algorithms for simulating optical properties of periodic metallic nanostructures arrays with the NHD model is another focus of future work.

\end{document}